\documentclass[11pt, two]{article}
\usepackage{mathrsfs}

\usepackage{CJK}
\usepackage{amsthm}
\usepackage{amsmath,amssymb,mathrsfs}
\usepackage{amsfonts}
\usepackage{yhmath}
\usepackage{graphics}
\usepackage{amssymb}
\usepackage{epsfig}
\usepackage{latexsym,bm}
\usepackage{indentfirst}
\usepackage[top=1in, bottom=1in, left=1.25in, right=1.25in]{geometry}
\numberwithin{equation}{section}

\newtheorem{theorem}{Theorem}[section]
\newtheorem{lemma}{Lemma}[section]
\newtheorem{proposition}{Proposition}[section]
\newtheorem{corollary}{Corollary}[section]
\newtheorem{remark}{Remark}[section]
\newtheorem{definition}{Definition}[section]
\numberwithin{equation}{section}

\def\be{\begin{equation}}
\def\ee{\end{equation}}
\def\br{\begin{eqnarray}}
\def\er{\end{eqnarray}}
\def\p{\partial}
\def\l{\langle}
\def\r{\rangle}
\def\w{\widetilde}
\def\f{\frac}

\begin{document}

\title{On Harmonic and Pseudoharmonic Maps from Strictly Pseudoconvex CR Manifolds}

\author{ $\mbox{Tian \ Chong }$ \hspace{12pt} $\mbox{Yuxin \ Dong }$  \hspace{12pt} $\mbox{Yibin \ Ren }$ \hspace{12pt}
$\mbox{Guilin \ Yang}$}
\date{}

\maketitle
\renewcommand{\thefootnote}{\fnsymbol{footnote}}

\begin{abstract}{Abstract.}
In this paper, we give some rigidity results for both harmonic and pseudoharmonic maps from CR manifolds into Riemannian manifolds or K\"ahler manifolds. Some basicity, pluriharmonicity and Siu-Sampson type results are established for both harmonic maps and pseudoharmonic maps.
\end{abstract}

\section{Introduction}
In 1980, Siu \cite{Siu1} studied the strong rigidity of compact K\"{a}hler manifolds by using the theory of harmonic maps. The basic discovery by Siu was a new Bochner-type formula for harmonic maps between K\"ahler manifolds, which does not involve the Ricci curvature tensor of the domains. Using the modified Bochner formula, he proved that any harmonic maps from a compact K\"ahler manifold to a K\"ahler manifold with strongly semi-negative curvature are actually pluriharmonic and some curvature terms of the pull-back complexifed tangent bundles vanish. When the target manifolds are K\"ahler manifolds with strongly negative curvature or compact quotients of irreducible bounded symmetric domains, the vanishing curvature terms, under the assumption of sufficiently high rank, force the maps to be either holomorphic or anti-holomorphic. Later, Sampson \cite{Sam2} showed that any harmonic maps from compact K\"ahler manifolds into Riemannian manifolds with nonpositive Hermitian curvature are also pluriharmonic, which generalized the pluriharmonicity result of Siu to more general targets. Pluriharmonic maps, holomorphic maps and Siu-Sampson type results have many important applications in geometry and topology of K\"ahler manifolds. The readers are refered to \cite{Tol} for details.

In 2002, Petit \cite{Pet} established some rigidity results for harmonic maps from strictly pseudoconvex CR manifolds to K\"ahler manifolds and Riemannian manifolds by using tools of Spinorial geometry. First, he proved that any harmonic map from a compact Sasakian manifold to a Riemannian manifold with nonpositive sectional curvature is trivial on the Reeb vector field. A map with this property will be called \emph{basic}. Next he proved that under suitable rank conditions the harmonic map from a compact Sasakian manifold to a K\"ahler manifold with strongly negative curvature is CR holomorphic or CR anti-holomorphic. However, it seems that Petit \cite{Pet} did not specifically discuss the relevant notions of pluriharmonicity. On the other hand, E. Barletta et al. in \cite{BDU} introduced the so-called pseudoharmonic maps from CR manifolds which are a natural generalization of harmonic maps. In his thesis \cite{Cha}, T.-H. Chang discussed some fundamental properties of pseudoharmonic maps.

In this paper, we will establish some rigidity results for both harmoinic maps and pseudoharmonic maps from CR manifolds by using the moving frame method. First, we find a result about the relationship between harmonic maps and pseudoharmonic maps from CR manifolds, which claims that these two kinds of maps are actually equivalent if the maps are basic. By the moving frame method, we not only recapture Petit's result about harmonic maps from compact Sasakian manifolds to Riemannian manifolds with nonpositive curvature (Proposition 5.1), but also show that the result is still valid for pseudoharmonic maps (Theorem 5.1).

The usual Bochner-type formula for the energy density of harmonic maps was given in \cite{EL}.
In \cite{Cha}, T.-H. Chang derived the CR Bochner-type formula for the pseudo-energy density of a pseudoharmonic map $\phi$ (Corollary 4.1). 
Unlike the Bochner formula of harmonic maps, there is a mixed term $i(\phi^i_{\alpha}\phi^i_{\bar{\alpha}0}-\phi^i_{\bar{\alpha}}\phi^i_{\alpha
0})$ appearing in the CR Bochner formula for the pseudoharmonic map. When $\phi$ is a function, it is known that the CR Paneitz operator, which is a divergence of a third order differential operator $P$, is a useful tool to treat such kind of term. One important property of the CR Paneitz operator is its nonnegativity when the dimension of the CR manifold $\geq 5$ (cf. \cite{CCh}). We generalize the operator $P$ to a differential operator, still denoted by $P$, acting on maps from a strictly pseudoconvex CR manifolds into a Riemannian manifold, and establish similar nonnegativity under the assumptions that the domain CR manifold has dimension $\geq 5$ and the target manifold is of nonpositive Hermitian curvature (Theorem 4.1). This enables us to establish a CR Bochner-type result for pseudohamonic maps (Theorem 4.2).

As mentioned previously, the notion of 'pluriharmonicity' is important for Siu-Sampson type results and other potential applications. We hope to disccuss suitable notion of pluriharmonic maps from CR manifolds. On a CR manifolds, we have two canonical connections, that is, the Levi-Civita connection of the Webster metric and the Tanaka-Webster connection of the pseudo-Hermitian structure. As a result, there are two kinds of second fundamental forms for a map from a CR manifold to a Riemannian manifold: the usual second fundamental form $B$ and a new second fundamental form $\beta$. The later one is defined with respect to the Tanaka-Webster connection of the domain CR manifold and the Levi-Civita connection
of the target Riemannian manifold (see Section 2). Using $B$, Ianus and Pastore \cite{IP1} defined two kinds of pluriharmonic notions. In \cite{DK}, Dragomir and Kamishima introduced the notion of CR pluriharmonic map by means of $\beta$. It turns out that a CR pluriharmonic map is basic and pseudoharmonic, and thus it is harmonic too. In addition, when the target manifold is K\"ahler, the CR pluriharmonic maps in \cite{DK} are more compatible with the CR holomorphic maps defined in \cite{GIP} in the sense that any CR holomorphic maps are automatically CR pluriharmonic. We also discuss the relationships between the CR pluriharmonic maps and those defined by Ianus and Pastore. Next, using the Siu-Sampson technique, we prove that any harmonic maps or pseudoharmonic maps from compact Sasakian manifolds to Riemannian manifolds with nonpositive Hermitian curvature or K\"ahler manifolds with strong semi-negative curvature are CR pluriharmonic (Theorems 6.1, 6.2). If the target is a K\"ahler manifold with strongly negative curvature and the rank of the map $\geq 3$ at some point, then the harmonic map or the pseduoharmonic map is CR holomorphic or CR anti-holomorphic (Theorem 7.2).
In \cite{Pet}, the author announced a similar result for harmonic maps using different technique.
When the target is a locally Hermitian symmetric space of noncompact type whose universal cover does not contain the hyperbolic plane as a factor, we show that the harmonic maps or pseudoharmonic maps are CR holomorphic under some explicit rank conditions (Theorem 7.1). These generalize some similar results in \cite{CT} to the CR case. To derive the above results, we also investigate the conic extensions of harmonic maps, CR pluriharmonic maps and CR holomorphic maps from Sasakian manifolds respectively, and establish also a unique continuation theorem for CR holomorphicity (Proposition 7.3). Using a technique in \cite{PRS}, we consider harmonic maps and pseudoharmonic maps from complete noncompact CR manifolds too. Under some decay conditions, some basicity and pluriharmonicity results are given.

Finally, we would like to mention that the second author \cite{Don2} has established similar rigidity results including Siu type results for pseudoharmonic maps between CR manifolds.

\begin{center}
TABLE OF CONTENTS\\
Abstract $\cdots\cdots\cdots\cdots\cdots\cdots\cdots\cdots\cdots\cdots\cdots\cdots\cdots\cdots\cdots\cdots\cdots\cdots$ $\cdot$ $1$\\
$1.$ Introduction $\cdots\cdots\cdots\cdots\cdots\cdots\cdots\cdots\cdots\cdots\cdots\cdots\cdots\cdots\cdots\cdots$ $\cdot$ $1$\\
$2.$ Preliminaries $\cdots\cdots\cdots\cdots\cdots\cdots\cdots\cdots\cdots\cdots\cdots\cdots\cdots\cdots\cdots\cdots$$\cdot$ $3$\\
$3.$ Commutative relations $\cdots\cdots\cdots\cdots\cdots\cdots\cdots\cdots\cdots\cdots\cdots\cdots\cdots$$\cdot$ $9$\\
$4.$ CR Bochner type results $\cdots\cdots\cdots\cdots\cdots\cdots\cdots\cdots\cdots\cdots\cdots\cdots$$\cdot$ $11$\\
$5.$ Basicity of harmonic and pseudoharmonic maps $\cdots \cdots\cdots\cdots\cdot$ $\cdot$ $15$\\
$6.$ CR pluriharmonicity of harmonic and pseudoharmonic maps $\cdot$$\cdot$ $19$\\
$7.$ Siu-Sampson type results $\cdots\cdots\cdots\cdots\cdots\cdots\cdots\cdots\cdots\cdots\cdots\cdots$$\cdot$ $23$\\
References $\cdots\cdots\cdots\cdots\cdots\cdots\cdots\cdots\cdots\cdots\cdots\cdots\cdots\cdots\cdots\cdots\cdots\cdot$ $\cdot$ $26$
\end{center}

\section{Preliminaries}
\subsection{Pseudohermitian structures}
A smooth manifold $M$ of real $(2m+1)$-dimension is said to be a CR manifold (of type (m,1)) if there exists a smooth $m$-dimensional complex subbundle $T_{1,0}M$ of the complexifed tangent bundle $T^{\mathbb{C}}M=TM\otimes\mathbb{C}$, such that
$$
T_{1,0}M\cap T_{0,1}M=\{0\}
$$
and
$$
[\Gamma^{\infty}(T_{1,0}M),\Gamma^{\infty}(T_{1,0}M)]\subseteq \Gamma^{\infty}(T_{1,0}M),
$$
where $T_{0,1}M=\overline{T_{1,0}M}$. The subbundle $T_{1,0}M$ is called a CR structure on $M$. Equivalently, the CR structure may also be described by the real subbundle
$H(M)=Re\{ T_{1,0}M \oplus T_{0,1}M\}$, which carries a complex structure $J:H(M)\rightarrow H(M)$ given by
$$J(Z+\bar{Z})=\sqrt{-1}(Z-\bar{Z})$$
for any $Z\in T_{1,0}M$.

Hereafter we assume $M$ is orientable. Set
$$ E_x=\{\omega\in T^*_xM: Ker(\omega)\supseteq H(M)_x\},$$
for any $x\in M$. Then $E\rightarrow M$ becomes an orientable real line subbundle of the cotangent bundle $T^*M$, and thus there exist globally defined nonvanishing sections $\theta\in \Gamma^{\infty}(E)$. Any such a section $\theta$ is called a pseudo-Hermitian structure on $M$. The Levi form $G_{\theta}$ of $\theta$ is defined by
$$G_{\theta}(X,Y)=d\theta(X,JY)$$ for any $X,Y\in H(M)$.
An orientable CR manifold endowed with a pseudo-Hermitian structure is called a pseudo-Hermitian manifold. A pseudo-Hermitian manifold $(M, J, \theta)$ is said to be a strictly pseudoconvex CR manifold if $L_{\theta}$ is positive definite. Standard examples for strictly pseudoconvex CR manifolds are the odd-dimensional spheres and the Heisenberg groups.

From now on, we always assume $(M,J,\theta)$ is strictly pseudoconvex. Consequently there exists a unique nonvanishing vector field $T$ on $M$, transverse to $H(M)$, satisfying $\theta(T)=1$ and $T\lrcorner d\theta=0$. The vector field $T$ is referred to as the characteristic direction or the Reeb field of $(M, J, \theta)$. Extending $J$ on $TM$ by $JT=0$, we can extend $L_{\theta}$ on $TM$ by the same formula as above.
This allows us to define a Riemannian metric $g_{\theta}$, called the Webster metric, as follows:
$$
g_{\theta}(X,Y)=G_{\theta}(\pi_HX,\pi_HY)+\theta(X)\theta(Y),
$$
for any $X,Y\in TM$, where $\pi_H:TM\rightarrow H(M)$ is the natural projection.
Then the two-form $\Omega$ defined by $\Omega(X,Y)=g_{\theta}(X,JY)$ coincides with the two-form $-d\theta$.

On a strictly pseudoconvex CR manifold, there exists a canonical connection preserving both the CR structure and the Webster metirc.

\begin{proposition}(cf. \cite{DT, Tan, Web})\label{pro201}
Let $(M, J, \theta)$ be a strictly pseudoconvex CR manifold and $g_{\theta}$ the Webster metric of $(M,J, \theta)$. Then there exists a unique linear connection $\nabla$ on $TM$, called the Tanaka-Webster connection, such that:\\
(1) the Levi distribution $H(M)$ is parallel with respect to $\nabla$;\\
(2) $\nabla g_{\theta}=0$, $\nabla J=0$, $\nabla \theta=0$ (hence $\nabla T=0$);\\
(3) the torsion $T_{\nabla}$ of $\nabla$ satisfies $T_{\nabla}(X,Y)=-\Omega(X,Y)T$ and $T_{\nabla}(T,JX)=-JT_{\nabla}(T,X)$, for any $X,Y\in H(M)$.
\end{proposition}

Unlike the Levi-Civita connection, the torsion $T_{\nabla}$ of the Tanaka-Webster connection $\nabla$ is always non zero. The $TM$-valued 1-form $\tau$, defined by $\tau(X)=T_{\nabla}(T,X)$, for any $X\in T(M)$, is called the pseudo-Hermitian torsion of $\nabla$. Note that $\tau$ is self-adjoint and trace-free with respect to the Webster metric $g_{\theta}$ (cf. Chapter 1 of \cite{DT}).

\begin{definition}
A strictly pseudoconvex CR manifold is called a Sasakian manifold if its pseudo-Hermitian torsion is zero.
\end{definition}

Choose a local orthonormal CR frame field $\{e_0=T, e_1, \cdots, e_m,Je_1,\cdots, Je_m\}$ on $M$. Set
$$
T_{\alpha}=\frac{1}{\sqrt{2}}(e_{\alpha}-\sqrt{-1}Je_{\alpha}), \quad   T_{\bar{\alpha}}=\frac{1}{\sqrt{2}}(e_{\alpha}+\sqrt{-1}Je_{\alpha}),
$$
then $\{T_{\alpha}\}$ is a local unitary frame of $T_{1,0}M$. Let $\{\theta, \theta^{\alpha},\theta^{\bar{\alpha}}\}$ be the dual frame field of $\{T, T_{\alpha},T_{\bar{\alpha}}\}$.
Clearly Proposition \ref{pro201} implies that there exist uniquely defined complex 1-forms $\theta^{\alpha}_{\beta}\in \Gamma^{\infty}(T^*M)\otimes\mathbb{C}$ such that
$$
\nabla T_{\alpha}=\theta_{\alpha}^{\beta}\otimes T_{\beta}, \quad  \nabla T_{\bar{\alpha}}=\theta_{\bar{\alpha}}^{\bar{\beta}}\otimes T_{\bar{\beta}},
$$
where $\theta^{\bar{\alpha}}_{\bar{\beta}}=\overline{\theta^{\alpha}_{\beta}}$. These are the connection 1-forms of the Tanaka-Webster connection $\nabla$.
Set $\tau (T_{\alpha})=A^{\bar{\beta}}_{\alpha}T_{\bar{\beta}}$, and $A(T_{\alpha}, T_{\beta})=g_{\theta}(\tau(T_{\alpha}), T_{\beta})=A_{\alpha \beta}$, then $A_{\alpha \beta}=A^{\bar{\gamma}}_\alpha \delta_{\gamma \beta}=A^{\bar{\beta}}_\alpha$.
We denote $\tau^{\alpha}=A^{\alpha}_{\bar{\beta}}\theta^{\bar{\beta}}$, then $\tau=\tau^{\alpha} \otimes T_{\alpha}+\tau^{\bar{\alpha}} \otimes T_{\bar{\alpha}}$.
Write $R_{\alpha\bar{\beta}\lambda\bar{\mu}}=g_{\theta}(R(T_{\lambda},T_{\bar{\mu}})T_{\alpha},T_{\bar{\beta}})=\delta_{\gamma\beta}R^{\beta}_{\alpha\lambda\bar{\mu}}$.

\begin{lemma}(cf. \cite{DT, Web})
The structure equations for the Tanaka-Webster connection of $(M,\theta,J)$ in terms of local orthonormal CR coframe field $\{\theta,\theta^{\alpha},\theta^{\bar{\alpha}}\}$ are
\br
\nonumber d\theta&=&\sqrt{-1}\delta_{\alpha\beta}\theta^{\alpha}\wedge\theta^{\bar{\beta}},\\
\label{pr1} d\theta^{\alpha}&=&\theta^{\beta}\wedge\theta_{\beta}^{\alpha}+\theta\wedge\tau^{\alpha}, \quad \theta_{\alpha}^{\beta}+\theta_{\bar{\beta}}^{\bar{\alpha}}=0,\\
\nonumber d\theta^{\alpha}_{\beta}&=&-\theta^{\alpha}_{\gamma}\wedge\theta^{\gamma}_{\beta}+\Pi^{\alpha}_{\beta},
\er
where
\br
\Pi^{\alpha}_{\beta}=R^{\alpha}_{\beta \gamma \bar{\delta}}\theta^{\gamma}\wedge\theta^{\bar{\delta}}+W^{\alpha}_{\beta \gamma}\theta^{\gamma}\wedge\theta-W^{\alpha}_{\beta \bar{\gamma}}\theta^{\bar{\gamma}}\wedge\theta+
\sqrt{-1}\theta_{\beta}\wedge\tau^{\alpha}-\sqrt{-1}\tau_{\beta}\wedge\theta^{\alpha},
\er
and
\br\label{pr3}
W^{\beta}_{\alpha \bar{\gamma}}=h^{\bar{\delta}\beta}A_{\bar{\gamma}\bar{\delta}, \alpha},\quad
W^{\beta}_{\alpha \gamma}=h^{\bar{\delta}\beta}A_{\alpha \gamma, \bar{\delta}}, \quad
\tau_{\alpha}=h_{\alpha\bar{\beta}}\tau^{\bar{\beta}}, \quad
\theta_{\alpha}=h_{\alpha\bar{\beta}}\theta^{\bar{\beta}},
\er
where $R$ denote the curvature tensor of $\nabla$.
\end{lemma}

From (\ref{pr1}), one may derive that (cf. \cite{Web}): $R_{\alpha\bar{\beta}\lambda\bar{\mu}}=R_{\lambda\bar{\beta}\alpha\bar{\mu}}$. The pseudo-Hermitian Ricci tensor is given by $R_{\lambda\bar{\mu}}=R^{\alpha}_{\lambda\alpha\bar{\mu}}=R^{\alpha}_{\alpha\lambda\bar{\mu}}$.

For a strictly pseudoconvex CR manifold $(M^{2m+1},J,\theta)$, we denote by $\nabla^{\theta}$ the Levi-Civita connection of the Webster metric $g_{\theta}$.
From Lemma 1.3 of \cite{DT}, we know the relation between the Tanaka-Webster connection $\nabla$ and Levi-Civita connection $\nabla^{\theta}$ of $(M,J,\theta)$:
\be
\label{pr4} \nabla^{\theta}=\nabla+(\frac{1}{2}\Omega-A)\otimes T+\tau\otimes \theta+\frac{1}{2}\theta\odot J,
\ee
where $A(X,Y)=g_{\theta}(\tau X,Y)$, $(\theta\odot J)(X,Y)=\theta(X)JY+\theta(Y)JX$ (cf. also \cite{LW}). By (\ref{pr4}), we have
$$
\nabla^{\theta}_X T=\tau(X)+\frac{1}{2}JX.
$$
In particular, $\nabla^{\theta}_T T=0$. If $X,Y\in H(M)$, then
\br\label{pr5}
\nabla^{\theta}_X Y=\nabla_X Y+[\frac{1}{2}\Omega(X,Y)-A(X,Y)]T.
\er

\begin{lemma}\label{prl2}
For any local orthonormal CR frame field $\{e_A\}_{A=0}^{2m}$, we have
\be
\sum_{A=0}^{2m} \nabla^{\theta}_{e_A} e_A=\sum_{A=0}^{2m} \nabla_{e_A} e_A.
\ee
In particular, we get
\be
\label{pr7}\sum_{A=1}^{2m} \nabla^{\theta}_{e_A} e_A=\sum_{A=1}^{2m} \nabla_{e_A} e_A.
\ee
\end{lemma}
\begin{proof}
By (\ref{pr4}), we have
$$\sum_{A=0}^{2n} \nabla^{\theta}_{e_A} e_A-\sum_{A=0}^{2n} \nabla_{e_A} e_A=-trace(\tau)T=0.$$
Since $\nabla^{\theta}_T T=\nabla_T T=0$, (\ref{pr7}) is valid.
\end{proof}

As a result of Lemma \ref{prl2}, we have
\begin{lemma}\label{prl3}
Let $(M,J,\theta)$ be a strictly pseudoconvex CR manifold and let $X$ be any vector field on $M$. Then
\be
\nonumber div X=\sum_{A=0}^{2n} g_{\theta}(\nabla_{e_A} X, e_A).
\ee
where $\nabla$ is the Tanaka-Webster connection of $M$ and $\{e_A\}_{A=0}^{2m}$ is a local orthonormal CR frame field on $M$. In particular, if $X\in H(M)$, then
$$
div X=\sum_{A=1}^{2n}  g_{\theta}(\nabla_{e_A} X, e_A).
$$
\end{lemma}

\subsection{Harmonic maps and pseudoharmonic maps}

Let $(M,J,\theta)$ be a strictly pseudoconvex CR manifold with the Tanaka-Webster connection $\nabla$ and let $(N,h)$ be a Riemannian manifold with Levi-Civita connection $\nabla^h$. For a smooth map $\phi:M\rightarrow N$, there are two induced connections $\nabla^{\theta}\otimes \phi^{-1}\nabla^h$ and $\nabla\otimes \phi^{-1}\nabla^h$ on $T^*M\otimes\phi^{-1}TN$. Using these two connections, one may define the usual second fundamental form $B$ and a new second fundamental form $\beta$ (cf. \cite{Pet}) for the map $\phi$ as follows:
\be
B(X,Y)=\nabla^h_Y (d\phi(X))-d\phi(\nabla^{\theta}_Y X)
\ee
and
\be\label{pr8}
\beta(X,Y)=\nabla^h_Y (d\phi(X))-d\phi(\nabla_Y X),
\ee
where $\phi^{-1}\nabla^h$ is written as $\nabla^h$ for simplicity. Due to Lemma \ref{prl2}, we have
\be\label{pr9}
trace_{g_{\theta}}B=trace_{g_{\theta}}\beta.
\ee
Recall that a map $\phi$ is called harmonic if $\tau^{\theta}(\phi):=trace_{g_{\theta}}B=0$ (cf. \cite{EL}). As a result of (\ref{pr9}), the harmonicity of $\phi$ can also be defined by $\beta$. Note that the most advantage of using $\nabla$ in (\ref{pr8}) is that the Tanaka-Webster connection preserves the CR structure; a little disadvantage of using $\nabla$ is that $\beta$ is no longer symmetric. However, we will see that the non-symmetry of $\beta$ may also lead to some unexpected geometric consequences.

For any bilinear form $C$ on $TM$, we denote by $\pi_H C$ the restriction of $C$ to $H(M)\otimes H(M)$.
\begin{definition}
A map $\phi:(M^{2m+1},J,\theta)\to (N,h)$ from a strictly pseudoconvex CR manifold to a Riemannian manifold is called a pseudoharmonic map if it is a critical point of the following pseudo-energy functional
\begin{equation}
E_{H}(\phi)=\int_Me_H(\phi)\Psi
\end{equation}
where $e_H(\phi)=\frac{1}{2}trace_{G_{\theta}}(\pi_H \phi^*h)$ is the pseudo-energy density of $\phi$ and $\Psi=\theta\wedge(d\theta)^m$ is the volume form of $g_{\theta}$.
\end{definition}

\begin{proposition}(cf. \cite{BDU, DT})
Let $\phi:(M^{2m+1},J,\theta)\to (N,h)$ be a smooth map from a strictly pseudoconvex CR manifold to a Riemannian manifold. Let $\tau(\phi)$ be pseudo-tensor field of $\phi$ defined by
\be\label{pr10}
\tau(\phi)=trace_{G_{\theta}}(\pi_H \beta).
\ee
Then $\phi$ is pseudoharmonic if and only if $\tau(\phi)=0$.
\end{proposition}

From Lemma \ref{prl2}, it is easy to see that
\be\label{pr11}
\tau(\phi)=trace_{G_{\theta}}(\pi_H B).
\ee

\begin{definition}
A smooth map $\phi:(M^{2m+1},J,\theta)\rightarrow (N,h)$ is called basic if $d\phi(T)=0$.
\end{definition}

\begin{proposition}\label{bal4}
Let $\phi:(M,J,\theta)\rightarrow (N,h)$ be a smooth map. Assume that $\nabla^{h}_T (d\phi(T))=0$, that is, $d\phi(T)$ is parallel in the direction $T$ with respect to the pull-back connection $\phi^{-1}\nabla^h$. Then $\tau^{\theta}(\phi)=\tau(\phi)$; and thus $\phi$ is harmonic if and onlu if $\phi$ is pseudoharmonic.
\end{proposition}
\begin{proof}
Choose a local orthonormal CR frame field $\{e_A\}_{A=0}^{2n}=\{T,e_1,e_2\cdots,e_{2n}\}$. Using Lemma \ref{prl2} and the assumption, we compute
\br
\nonumber \tau^{\theta}(\phi)&=&\sum_{A=1}^{2n}[\nabla^h_{e_A} (d\phi(e_A))-d\phi(\nabla^{\theta}_{e_A}e_A)]+\nabla^h_{T}(d\phi(T))\\
\nonumber &=&\sum_{A=1}^{2n}[\nabla^h_{e_A} (d\phi(e_A))-d\phi(\nabla_{e_A}e_A)]\\
&=&\tau(\phi).
\er
\end{proof}

\begin{corollary}\label{cor21}
Let $\phi:(M^{2m+1},J,\theta)\rightarrow (N,h)$ be a baisc map. Then $\phi$ is harmonic if and only if $\phi$ is pseudoharmonic.
\end{corollary}

\begin{definition}
Let $\phi:(M,J, \theta)\rightarrow (N,h)$ be a smooth map from a strictly pseudoconvex CR manifold into a Riemannian manifold. We say that\\
 (i) (\cite{IP1}) $\phi$ is J-pluriharmonic, if $B(X,Y)+B(JX,JY)=0$, for any $X,Y\in TM$\\
 (ii) (\cite{IP1}) $\phi$ is H-pluriharmonic, if $B(Z,W)+B(JZ,JW)=0$, for any $Z,W\in H(M)$;\\
 (iii)  $\phi$ is B-pluriharmonic, if $B(T,T)=0$ and $B(Z,W)+B(JZ,JW)=0$, for any $Z,W\in H(M)$;\\
 (iv) (\cite{DK}) $\phi$ is CR pluriharmonic, if $\beta(Z,W)+\beta(JZ,JW)=0$, for any $Z,W\in H(M)$;\\
 (v) (\cite{GIP}) When $(N,h)$ is a K\"{a}hler manifold with complex structure $J'$, $\phi$ is called a CR holomorphic (resp. CR anti-holomorphic) map, if
\be
d\phi\circ J=J'\circ d\phi, \quad  (resp.\quad d\phi\circ J=-J'\circ d\phi).
\ee
\end{definition}

\begin{remark}
(1) The concepts of J-pluriharmonic map and H-pluriharmonic map were introduced by Ianus and Pastore in \cite{IP1} where $J$ and $H(M)$ are denoted by $\varphi$ and $D$ respectively. And they proved that the J-pluriharmonic maps are harmonic. 

(2) Dragomir and Kamishima in \cite{DK} introduced the notion of CR pluriharmonic maps under the name of $\bar{\partial}$-pluriharmonic map, and then they proved that every CR pluriharmonic map is a pseudoharmonic map.

(3) In \cite{GIP} the authors introduced the notion of CR holomorphic map under the name of the $(J,J')$-holomorphic map. They proved that the CR holomorphic map is harmonic. If $(M,g,J)$ is a K\"ahler manifold, $(N,h)$ is a Riemannian manifold and the map $\phi:M \rightarrow N$ satisfies
$$
B(X,Y)+B(JX,JY)=0,
$$
for any $X,Y\in TM$, then the map $\phi$ is called a pluriharmonic map (cf. \cite{Don1, Uda}).
\end{remark}

Obviously, J-pluriharmonicity implies B-pluriharmonicity, and B-pluriharmonicity implies H-pluriharmonicity.
Both J-pluriharmonic maps and B-pluriharmonic maps are harmonic.
By (\ref{pr10}) and (\ref{pr11}), both the CR pluriharmonic map and the H-pluriharmonic map are pseudoharmonic.

\begin{proposition}\label{bap3}
(i) (cf. \cite{DK}) If $\phi:(M,J,\theta)\rightarrow (N,h)$ is CR pluriharmonic, then $\phi$ is a basic and pseudoharmonic map. Moreover, $\phi$ is B-pluriharmonic too.\\
(ii) If $\phi:(M,J,\theta)\rightarrow (N,h)$ is baisc and H-pluriharmonic, then $\phi$ is CR pluriharmonic.
\end{proposition}
\begin{proof}
(i)
For any $Z=X-\sqrt{-1}JX, W=Y-\sqrt{-1}JY\in T_{1,0}M$, we have
\br
\beta(Z,\overline{W})=\beta(X,Y)+\beta(JX,JY)+\sqrt{-1}[\beta(X,JY)-\beta(JX,Y)],
\er
thus we get that $\phi$ is CR pluriharmonic if and only if $(\pi_H\beta)^{(1,1)}=0$. Thus the CR pluriharmonic map is pseudoharmonic.

On the other hand, we have
\br
\nonumber 0&=&\beta(Z,\overline{W})-\beta(\overline{W},Z)\\
\nonumber &=&d\phi(T_{\nabla}(Z,\overline{W}))\\
\nonumber &=&-\Omega(Z,\overline{W})d\phi(T)\\
&=&\sqrt{-1} g_{\theta}(Z,\overline{W})d\phi(T).
\er
If we take $Z=W\neq 0$, then $g_{\theta}(Z,\overline{W})\neq 0$, thus we have $d\phi(T)=0$.

For any $X,Y\in H(M)$, by (\ref{pr5}) and $A(JY,JX)=A(Y,X)$, we have
\br\label{y212}
B(X,Y)+B(JX,JY)=\beta(X,Y)+\beta(JX,JY)-\Omega(Y,X)d\phi(T).
\er
Thus if $\phi$ is baisc, then 
the CR pluriharmonic map $\phi$ is B-pluriharmonic.

(ii) This can be proved by (\ref{y212}).
\end{proof}

\begin{proposition}\label{bap4}
Suppose $\phi: (M,J,\theta)\rightarrow (N,h,J')$ is a CR $\pm$holomorphic map from a strictly pseudoconvex CR manifold $M$ into a K\"{a}hler manifold $N$. Then $\phi$ is CR pluriharmonic.
\end{proposition}
\begin{proof}
Suppose $\phi$ is CR holomorphic map.
For any $X,Y\in H(M)$, we have
\br
\nonumber \beta(JX,Y)&=&\nabla^h_Y(\phi(JX))-d\phi(\nabla_Y JX)\\
\nonumber &=&\nabla^h_Y(J'd\phi(X))-d\phi(J(\nabla_Y X))\\
\nonumber &=&J'\nabla^h_Y (d\phi(X))-J'd\phi(\nabla_Y X)\\
&=&J'\beta(X,Y).
\er
Since $J'd\phi(T)=d\phi(JT)=0$, we get that $\phi$ is baisc. Because of $\beta(X,Y)-\beta(Y,X)=-\Omega(X,Y)d\phi(T)$, we have that $\beta$ is symmetric on $H(M)\otimes H(M)$. Thus we derive
\be
\nonumber \beta(JX,JY)=J'\beta(X,JY)=J'\beta(JY,X)=-\beta(Y,X)=-\beta(X,Y).
\ee
Therefore, the map $\phi$ is CR pluriharmonic.
If $\phi$ is CR anti-holomorphic map, the conclusion can be proved in a similar way.
\end{proof}

\section{Commutative relations}
Let $\phi:(M^{2m+1},J, \theta)\rightarrow (N^n, h)$ be a smooth map from a strictly pseudoconvex CR manifold into a Riemannian manifold. Choose a local orthonormal CR coframe field $\{\theta, \theta^{\alpha}, \theta^{\bar{\alpha}}\}$ on $M$ and a local orthonormal coframe field
$\{\omega^i\}$ on $N$. Throughout this paper we will employ the index conventions
$$A,B,C=0,1,\cdots,m,\bar{1},\cdots,\bar{m},$$
$$\alpha,\beta,\gamma=1,\cdots,m,$$
$$i,j,k=1,\cdots,n,$$
and use the summation convention on repeating indices. The structure equations for the Riemannian connection of $(N,h)$ in terms of local orthonormal frame $\{\omega^i\}$ are
\br
\nonumber d\omega^i=-\omega^i_j\wedge\omega^j,  \quad  \omega^i_j+\omega^j_i=0,\\
\label{cr1} d\omega^i_j=-\omega^i_k\wedge\omega^k_j+\Omega^i_j,
\er
where $\Omega^i_j=\frac{1}{2}\widetilde{R^i_{jkl}}\omega^k\wedge\omega^l$ are the components of the curvature form of $\nabla^h$.

Under the map $\phi:M\rightarrow N$, we have
\be\label{cr3}
\phi^*\omega^i=\phi^i_{\alpha}\theta^{\alpha}+\phi^i_{\bar{\alpha}}\theta^{\bar{\alpha}}+\phi^i_0\theta.
\ee
Hereafter we will drop $\phi^*$ in such formulas when their meaning are clear from context.
By taking the exterior derivative of (\ref{cr3}) and making use of the structure equations (\ref{pr1})-(\ref{pr3}) and (\ref{cr1}), we get
\br\label{cr4}
D\phi^i_B\wedge \theta^B+\sqrt{-1}\phi^i_0\theta^\alpha\wedge\theta^{\bar{\alpha}}-\phi^i_\alpha A_{\bar{\alpha}\bar{\beta}}\theta^{\bar{\beta}}\wedge\theta-\phi^i_{\bar{\alpha}} A_{\alpha\beta}\theta^\beta\wedge\theta=0,
\er
where
\br
\label{cr5}
D\phi^i_{\alpha}&=&d\phi^i_\alpha-\phi^i_\beta\theta^\beta_\alpha+\phi^j_\alpha\omega^i_j=\phi^i_{\alpha B}\theta^B,\\
\label{cr6}D\phi^i_{\bar{\alpha}}&=&d\phi^i_{\bar{\alpha}}-\phi^i_{\bar{\beta}}\theta^{\bar{\beta}}_{\bar{\alpha}}+\phi^j_{\bar{\alpha}}\omega^i_j=\phi^i_{\bar{\alpha} B}\theta^B,\\
\label{cr7}D\phi^i_0&=&d\phi^i_0+\phi^j_0\omega^i_j=\phi^i_{0B}\theta^B.
\er
From (\ref{cr4}) it follows that
\br\label{cr8}
\phi^i_{\alpha \beta}=\phi^i_{\beta \alpha},\quad
\phi^i_{\alpha\bar{\beta}}-\phi^i_{\bar{\beta}\alpha}=\sqrt{-1}\delta_{\alpha\beta}\phi^i_0,\quad
\phi^i_{0\alpha}-\phi^i_{\alpha 0}=\phi^i_{\bar{\beta}}A_{\beta\alpha}.
\er
Then the map $\phi$ is harmonic if and only if
\be\nonumber
\phi^i_{\alpha\bar{\alpha}}+\phi^i_{\bar{\alpha}\alpha}+\phi^i_{00}=0,
\ee
and $\phi$ is pseudoharmonic if and only if
\be\nonumber
\phi^i_{\alpha\bar{\alpha}}+\phi^i_{\bar{\alpha}\alpha}=0.
\ee

Differentiating the equation (\ref{cr5}) and using the structure equations in $M$ and $N$, we have
\be\label{cr9}
D\phi^i_{\alpha B}\wedge \theta^{B}+\sqrt{-1}\phi^i_{\alpha 0}\theta^{\beta}\wedge\theta^{\bar{\beta}}-\phi^i_{\alpha\beta}A_{\bar{\beta}\bar{\gamma}}\theta^{\bar{\gamma}}\wedge\theta-\phi^i_{\alpha\bar{\beta}}A_{\beta\gamma}\theta^{\gamma}\wedge\theta=-\phi^i_{\beta}\Pi^{\beta}_{\alpha}+\phi^j_{\alpha}\Omega^i_j,
\ee
where
\br\nonumber
D\phi^i_{\alpha \beta}&=&d\phi^i_{\alpha\beta}-\phi^i_{\alpha\gamma}\theta^\gamma_\beta-\phi^i_{\gamma\beta}\theta^\gamma_\alpha+\phi^j_{\alpha\beta}\omega^i_j=\phi^i_{\alpha\beta B}\theta^B,\\
\nonumber
D\phi^i_{\alpha \bar{\beta}}&=&d\phi^i_{\alpha\bar{\beta}}-\phi^i_{\alpha\bar{\gamma}}\theta^{\bar{\gamma}}_{\bar{\beta}}-\phi^i_{\gamma \bar{\beta}}\theta^\gamma_\alpha+\phi^j_{\alpha\bar{\beta}}\omega^i_j=\phi^i_{\alpha\bar{\beta}B}\theta^B,\\
\nonumber
D\phi^i_{\alpha 0}&=&d\phi^i_{\alpha0}-\phi^i_{\gamma 0}\theta^{\gamma}_{\alpha}+\phi^j_{\alpha 0}\omega^i_j=\phi^i_{\alpha0B}\theta^B.
\er
From (\ref{cr9}), we get the following commutative relations
\br
\label{cr10}
\phi^i_{\alpha \beta \gamma}&=&\phi^i_{\alpha \gamma \beta}-\phi^j_{\alpha}\phi^k_{\beta}\phi^l_{\gamma}\widehat{R^i_{jkl}}+
\sqrt{-1}\phi^i_{\beta}A_{\alpha\gamma}-\sqrt{-1}\phi^i_{\gamma}A_{\alpha\gamma},\\
\label{cr16}
\phi^i_{\alpha \bar{\beta} \bar{\gamma}}&=&\phi^i_{\alpha \bar{\gamma} \bar{\beta}}-\phi^j_{\alpha}\phi^k_{\bar{\beta}}\phi^l_{\bar{\gamma}}\widehat{R^i_{jkl}}+
\sqrt{-1}\delta_{\alpha\beta}\phi^i_{\lambda}A_{\bar{\lambda}\bar{\gamma}}-\sqrt{-1}\delta_{\alpha \gamma}\phi^i_{\lambda}A_{\bar{\lambda}\bar{\beta}},\\
\label{cr17}
\phi^i_{\alpha \beta \bar{\gamma}}&=&\phi^i_{\alpha \bar{\gamma} \beta}-\phi^j_{\alpha}\phi^k_{\beta}\phi^l_{\bar{\gamma}}\widehat{R^i_{jkl}}+\phi^i_{\lambda}R^{\lambda}_{\alpha \beta \bar{\gamma}}+\sqrt{-1}\delta_{\beta \gamma}\phi^i_{\alpha 0},\\
\phi^i_{\alpha \beta 0}&=&\phi^i_{\alpha 0 \beta}-\phi^j_{\alpha}\phi^k_{\beta}\phi^l_0\widehat{R^i_{jkl}}+\phi^i_{\gamma}A_{\alpha \beta,\gamma}-\phi^i_{\alpha \bar{\gamma}}A_{\gamma\beta},\\
\label{cr11}
\phi^i_{\alpha \bar{\beta} 0}&=&\phi^i_{\alpha 0 \bar{\beta}}-\phi^j_{\alpha}\phi^k_{\bar{\beta}}\phi^l_0\widehat{R^i_{jkl}}-\phi^i_{\gamma}A_{\bar{\beta} \bar{\gamma},\alpha}-\phi^i_{\alpha \gamma}A_{\bar{\gamma}\bar{\beta}},
\er
where $\widehat{R^i_{jkl}}=\widetilde{R^i_{jkl}}\circ\phi$.

Since the formula (\ref{cr6}) is the complex conjugate of (\ref{cr5}), then, after taking the exterior derivative of (\ref{cr6}) and using the structure equations, we find that the complex conjugate of formulas (\ref{cr10})-(\ref{cr11}) are valid too.

Similarly the exterior derivative of (\ref{cr7}) yields that
\be\label{cr15}
D\phi^i_{0B}\wedge\theta^B+\sqrt{-1}\phi^i_0\theta^\alpha\wedge\theta^{\bar{\alpha}}-\phi^i_\alpha A_{\bar{\alpha}\bar{\beta}}\theta^{\bar{\beta}}\wedge\theta-\phi^i_{\bar{\alpha}} A_{\alpha\beta}\theta^\beta\wedge\theta=\phi^j_0\Omega^i_j,
\ee
where
\br\nonumber
D\phi^i_{0\alpha}&=&d\phi^i_{0\alpha}-\phi^i_{0\beta}\theta^{\beta}_{\alpha}+\phi^j_{0\alpha}\omega^i_j=\phi^i_{0\alpha B}\theta^B,\\
\nonumber
D\phi^i_{0\bar{\alpha}}&=&d\phi^i_{0\bar{\alpha}}-\phi^i_{0\bar{\beta}}\theta^{\bar{\beta}}_{\bar{\alpha}}+\phi^j_{0\bar{\alpha}}\omega^i_j=\phi^i_{0\bar{\alpha}B}\theta^B,\\
\nonumber
D\phi^i_{00}&=&d\phi^i_{00}+\phi^j_{00}\omega^i_j=\phi^i_{00B}\theta^B.
\er
We get from (\ref{cr15}) the commutative relations
\br
\phi^i_{0 \alpha \beta}&=&\phi^i_{0 \beta \alpha}-\phi^j_0
\phi^k_{\alpha}\phi^l_{\beta}\widehat{R^i_{jkl}},\\
\phi^i_{0 \alpha \bar{\beta}}&=&\phi^i_{0 \bar{\beta} \alpha}-\phi^j_0 \phi^k_{\alpha} \phi^l_{\bar{\beta}}\widehat{R^i_{jkl}}+\sqrt{-1}\delta_{\alpha \beta}\phi^i_{00},\\
\phi^i_{00\alpha}&=&\phi^i_{0 \alpha 0}-\phi^j_0 \phi^k_0 \phi^l_{\alpha}\widehat{R^i_{jkl}}+\phi^i_{0 \bar{\beta}}A_{\beta \alpha}.
\er

From (\ref{cr8}), we can derive:
\br\label{cr18}
\phi^i_{\alpha \bar{\beta} \gamma}&=&\phi^i_{\bar{\beta} \alpha \gamma}+\sqrt{-1}\delta_{\alpha\beta}\phi^i_{0 \gamma},\\
\label{cr19}
\phi^i_{\alpha \bar{\beta} \bar{\gamma}}&=&\phi^i_{\bar{\beta} \alpha \bar{\gamma}}+\sqrt{-1}\delta_{\alpha\beta}\phi^i_{0 \bar{\gamma}},\\
\phi^i_{0 \alpha \beta}&=&\phi^i_{\alpha 0 \beta}+\phi^i_{\bar{\gamma}\beta}A_{\gamma\alpha}+\phi^i_{\bar{\gamma}}A_{\gamma\alpha,\beta},\\
\label{cr21}\phi^i_{0 \alpha \bar{\beta}}&=&\phi^i_{\alpha 0 \bar{\beta}}+\phi^i_{\bar{\gamma}\bar{\beta}}A_{\gamma\alpha}+\phi^i_{\bar{\gamma}}A_{\gamma\alpha,\bar{\beta}}.
\er

If $(N,h)$ is a K\"{a}hler manifold, we choose a local orthonormal coframe field $\{\omega^i,\omega^{\bar{i}}\}$ on $N$.
The structure equations for the Riemannian connection of $(N,h)$ in terms of local orthonormal frame $\{\omega^i,\omega^{\bar{i}}\}$ are
\br
\begin{aligned}
d\omega^i=-\omega^i_j\wedge\omega^j,  \quad  \omega^i_j+\omega^{\bar{j}}_{\bar{i}}=0,\\
d\omega^i_j=-\omega^i_k\wedge\omega^k_j+\Omega^i_j,
\end{aligned}
\er
where $\Omega^i_j=\widetilde{R^i_{jk\overline{l}}} \omega^k\wedge\omega^{\bar{l}}$. Similar to the above discussions, we may obtain the following commutative formula:
\br\label{cr23}
\phi^i_{\alpha \bar{\beta} \bar{\gamma}}&=&\phi^i_{\alpha \bar{\gamma} \bar{\beta}}-\phi^j_{\alpha}\phi^k_{\bar{\beta}}\phi^{\bar{l}}_{\bar{\gamma}}\widehat{R^i_{jk\overline{l}}}+\phi^j_{\alpha}\phi^k_{\bar{\gamma}}\phi^{\bar{l}}_{\bar{\beta}}\widehat{R^i_{jk\overline{l}}} +\sqrt{-1}\delta_{\alpha\beta}\phi^i_{\lambda}A_{\bar{\lambda}\bar{\gamma}}-\sqrt{-1}\delta_{\alpha \gamma}\phi^i_{\lambda}A_{\bar{\lambda}\bar{\beta}}.
\er

\section{CR Bochner type result}

Let $(M^{2m+1},J,\theta)$ be a compact strictly pseudoconvex CR manifold. In \cite{GL, Lee} the authors introduced the following differential operator acting on functions
$$
Pf=\sum(f_{\bar{\alpha}\alpha\beta}+\sqrt{-1}mA_{\beta \alpha}f_{\bar{\alpha}})\theta^{\beta}=(P_{\beta}f)\theta^{\beta},
$$
which charecterizes CR pluriharmonic functions on $M$. In \cite{CCh} S.-C. Chang and H.-L.Chiu discussed the CR Paneitz operator
$$
P_0f=4[\delta_b(Pf)+\bar{\delta_b}(\bar{P}f)],
$$
where $\delta_b$ is the divergence operator that take $(1,0)$-forms to functions, and they proved that when $m\geq 2$, the corresponding CR Paneitz operator is always nonnegative, that is
$$\int_M P_0f\cdot f\Psi\geq 0,$$
where $\Psi$ is the volume form of $g_{\theta}$.

Now we want to generalize the operator $P$ to an operator, still denoted by $P$, acting on maps from strictly pseudoconvex CR manifolds into Riemannian manifolds. We will establish similar nonnegative property for the generalized operator P under suitable condition.
Suppose $\phi:(M^{2m+1},\theta,J)\rightarrow (N^n, h)$ is a smooth map from a strictly pseudoconvex CR manifold $M$ into a Riemannian manifold $N$.
We choose a local orthonormal CR coframe field $\{\theta,\theta^{\alpha},\theta^{\bar{\alpha}}\}$ on $M$, a local orthonormal frame field $\{E_{i}\}$ on $N$.
We still use the notaions of the last section. Define
$$
P\phi=(P^j_{\beta}\phi)\theta^{\beta}\otimes E_j,
$$
where $P^j_{\beta}\phi=\phi^j_{\bar{\alpha}\alpha\beta}+\sqrt{-1}mA_{\beta\alpha}\phi^j_{\bar{\alpha}}$.

Let
\be
\theta_{W_1}=\phi^i_{\alpha}\phi^i_{\bar{\alpha}\bar{\beta}}\theta^{\bar{\beta}}+\phi^i_{\bar{\alpha}}\phi^i_{\alpha\beta}\theta^{\beta}+\phi^i_{\alpha}\phi^i_{\bar{\alpha}\beta}\theta^{\beta}+\phi^i_{\bar{\alpha}}\phi^i_{\alpha\bar{\beta}}\theta^{\bar{\beta}}.
\ee
Evidently the 1-form $\theta_{W_1}$, which is a well-defined on $M$, is the 1-form corresponding to the horizontal gradient $\nabla^H(e_H(\phi))$ of $e_H(\phi)$, where $\nabla^H(e_H(\phi))=\Pi_H\nabla(e_H(\phi))$ and $g_{\theta}(\nabla e_H(\phi),X)=X(e_H(\phi))$ for any $X\in \chi(M)$.

\begin{lemma}
Set $\w{R_{ijkl}}=g_{ip}\w{R^p_{jkl}}=\delta_{ip}\w{R^p_{jkl}}=\w{R^i_{jkl}}$. Then
\begin{eqnarray}
\nonumber div\theta_{W_1}&=&2(|\phi^i_{\alpha\beta}|^2+|\phi^i_{\alpha\bar{\beta}}|^2)+\l\l d_b\phi, \nabla_b \tau(\phi) \r\r +2\phi^i_{\alpha}\phi^i_{\bar{\beta}}Ric_{\bar{\alpha}\beta}\\
\nonumber &&-\sqrt{-1}m(\phi^i_{\alpha}\phi^i_{\beta}A_{\bar{\alpha}\bar{\beta}}-\phi^i_{\bar{\alpha}}\phi^i_{\bar{\beta}}A_{\alpha\beta})
-2\sqrt{-1}(\phi^i_{\alpha}\phi^i_{\bar{\alpha}0}-\phi^i_{\bar{\alpha}}\phi^i_{\alpha0})\\
\label{bo3}
&&-2(\phi^i_{\bar{\alpha}}\phi^j_{\beta}\phi^k_{\alpha}\phi^l_{\bar{\beta}}\widehat{R_{ijkl}}
+\phi^i_{\alpha}\phi^j_{\beta}\phi^k_{\bar{\alpha}}\phi^l_{\bar{\beta}}\widehat{R_{ijkl}}),
\end{eqnarray}
where $\nabla_b\tau(\phi)=(\phi^i_{\alpha\bar{\alpha}\beta}+\phi^i_{\bar{\alpha}\alpha\beta})\theta^{\beta}\otimes E_i+(\phi^i_{\alpha\bar{\alpha}\bar{\beta}}+\phi^i_{\bar{\alpha}\alpha\bar{\beta}})\theta^{\bar{\beta}}\otimes E_i$, and $\l\l \cdot, \cdot\r\r$ is the metric in $T^*M\otimes \phi^{-1}TN$ induced by $g_{\theta}$ and $h$.
\end{lemma}
\begin{proof}
Using Lemma \ref{prl3} and the commutative relations in Section 3, we compute
\begin{eqnarray*}
div \theta_{W_1}&=&(\phi^i_{\alpha}\phi^i_{\bar{\alpha}\bar{\beta}}),_{\beta}+(\phi^i_{\bar{\alpha}}\phi^i_{\alpha\beta}),_{\bar{\beta}}
+(\phi^i_{\alpha}\phi^i_{\bar{\alpha}\beta}),_{\bar{\beta}}+(\phi^i_{\bar{\alpha}}\phi^i_{\alpha\bar{\beta}}),_{\beta}\\ &=&\phi^i_{\alpha\beta}\phi^i_{\bar{\alpha}\bar{\beta}}+\phi^i_{\alpha}\phi^i_{\bar{\alpha}\bar{\beta}\beta}+\phi^i_{\bar{\alpha}\bar{\beta}}\phi^i_{\alpha\beta}
+\phi^i_{\bar{\alpha}}\phi^i_{\alpha\beta\bar{\beta}}+\phi^i_{\alpha\bar{\beta}}\phi^i_{\bar{\alpha}\beta}+\phi^i_{\alpha}\phi^i_{\bar{\alpha}\beta\bar{\beta}}\\
&&+\phi^i_{\bar{\alpha}\beta}\phi^i_{\alpha\bar{\beta}}+\phi^i_{\bar{\alpha}}\phi^i_{\alpha\bar{\beta}\beta}\\
&=&2(|\phi^i_{\alpha\beta}|^2+|\phi^i_{\alpha\bar{\beta}}|^2)+\l\l d_b\phi, \nabla_b \tau(\phi) \r\r +2\phi^i_{\alpha}\phi^i_{\bar{\beta}}Ric_{\bar{\alpha}\beta}-\sqrt{-1}m(\phi^i_{\alpha}\phi^i_{\beta}A_{\bar{\alpha}\bar{\beta}}\\
&&-\phi^i_{\bar{\alpha}}\phi^i_{\bar{\beta}}A_{\alpha\beta})-2\sqrt{-1}(\phi^i_{\alpha}\phi^i_{\bar{\alpha}0}-\phi^i_{\bar{\alpha}}\phi^i_{\alpha0})
-2(\phi^i_{\bar{\alpha}}\phi^j_{\beta}\phi^k_{\alpha}\phi^l_{\bar{\beta}}\widehat{R_{ijkl}}
+\phi^i_{\alpha}\phi^j_{\beta}\phi^k_{\bar{\alpha}}\phi^l_{\bar{\beta}}\widehat{R_{ijkl}}).
\end{eqnarray*}
\end{proof}

From (\ref{cr8}), (\ref{cr17}) and (\ref{cr19}), we get immediately the following Lemmas 4.2, 4.3 and 4.4.
\begin{lemma}
\br
\nonumber \sqrt{-1}(\phi^i_{\alpha}\phi^i_{\bar{\alpha}0}-\phi^i_{\bar{\alpha}}\phi^i_{\alpha 0})&=&\frac{2}{m}\l\l P\phi+\overline{P\phi}, d_b\phi\r\r-\frac{1}{m}\l\l d_b\phi,\nabla_b\tau(\phi)\r\r\\
\label{bo4}& &+\sqrt{-1}(\phi^i_{\alpha}\phi^i_{\beta}A_{\bar{\alpha}\bar{\beta}}-\phi^i_{\bar{\alpha}}\phi^i_{\bar{\beta}}A_{\alpha\beta}).
\er
\end{lemma}

Thus we have
\begin{corollary}
\begin{eqnarray}
\nonumber div\theta_{W_0}&=&2(|\phi^i_{\alpha\beta}|^2+|\phi^i_{\alpha\bar{\beta}}|^2)+(1+\frac{2}{m})\l\l d_b\phi, \nabla_b \tau(\phi) \r\r +2\phi^i_{\alpha}\phi^i_{\bar{\beta}}Ric_{\bar{\alpha}\beta}\\
\nonumber &&-\sqrt{-1}(m+2)(\phi^i_{\alpha}\phi^i_{\beta}A_{\bar{\alpha}\bar{\beta}}
-\phi^i_{\bar{\alpha}}\phi^i_{\bar{\beta}}A_{\alpha\beta})-\frac{4}{m}\l\l P\phi+\overline{P\phi}, \nabla_b \phi \r\r\\
\label{bo31}
&&-2(\phi^i_{\bar{\alpha}}\phi^j_{\beta}\phi^k_{\alpha}\phi^l_{\bar{\beta}}\widehat{R_{ijkl}}
+\phi^i_{\alpha}\phi^j_{\beta}\phi^k_{\bar{\alpha}}\phi^l_{\bar{\beta}}\widehat{R_{ijkl}}),
\end{eqnarray}
\end{corollary}

\begin{remark}
Since $div\theta_{W_1}=\Delta_b(e_H(\phi))$, the formula (\ref{bo3}) and (\ref{bo31}) are both called the CR Bochner formulae.
\end{remark}

Integrating both sides of (\ref{bo4}) and using the divergence theorem, we get
\br
\label{bo7}
\nonumber \sqrt{-1}\int_M (\phi^i_{\alpha}\phi^i_{\bar{\alpha}0}-\phi^i_{\bar{\alpha}}\phi^i_{\alpha 0})\Psi&=&\frac{2}{m}\int_M \l\l P\phi+\overline{P\phi}, d_b\phi\r\r\Psi+\frac{1}{m}\int_M |\tau(\phi)|^2\Psi\\
&&+\sqrt{-1}\int_M (\phi^i_{\alpha}\phi^i_{\beta}A_{\bar{\alpha}\bar{\beta}}-\phi^i_{\bar{\alpha}}\phi^i_{\bar{\beta}}A_{\alpha\beta})\Psi.
\er

\begin{lemma}
\be\label{bo81}
\sqrt{-1}\int_M (\phi^i_{\alpha}\phi^i_{\bar{\alpha}0}-\phi^i_{\bar{\alpha}}\phi^i_{\alpha 0})\Psi=m\int_M(\phi^i_0)^2\Psi-
\sqrt{-1}\int_M(\phi^i_{\alpha}\phi^i_{\beta}A_{\bar{\alpha}\bar{\beta}}-\phi^i_{\bar{\alpha}}\phi^i_{\bar{\beta}}A_{\alpha\beta})\Psi.
\ee
\end{lemma}

\begin{lemma}
\br
\nonumber
2\int_M \phi^i_{\alpha}\phi^i_{\bar{\beta}}Ric_{\bar{\alpha}\beta}\Psi&=&-2\int_M (|\phi^i_{\alpha\beta}|^2-|\phi^i_{\alpha\bar{\beta}}|^2)\Psi+\sqrt{-1}m\int_M (\phi^i_{\alpha}\phi^i_{\bar{\alpha}0}-\phi^i_{\bar{\alpha}}\phi^i_{\alpha 0})\Psi\\
\label{bo8}&&+2\int_M \phi^i_{\alpha}\phi^j_{\bar{\alpha}}\phi^k_{\bar{\beta}}\phi^l_{\beta}\widehat{R_{ijkl}}\Psi.
\er
\end{lemma}

Integrating (\ref{bo3}) on $M$ and substituting (\ref{bo8}) into it, we have
\br
\nonumber
0&=&4\int_M |\phi^i_{\alpha\bar{\beta}}|^2\Psi-\int_M |\tau(\phi)|^2\Psi+\sqrt{-1}(m-2)\int_M (\phi^i_{\alpha}\phi^i_{\bar{\alpha}0}-\phi^i_{\bar{\alpha}}\phi^i_{\alpha 0})\Psi\\
\nonumber &&-\sqrt{-1}m\int_M(\phi^i_{\alpha}\phi^i_{\beta}A_{\bar{\alpha}\bar{\beta}}-\phi^i_{\bar{\alpha}}\phi^i_{\bar{\beta}}A_{\alpha\beta})\Psi\\
\nonumber &&-2\int_M (\phi^i_{\bar{\alpha}}\phi^j_{\beta}\phi^k_{\alpha}\phi^l_{\bar{\beta}}\widehat{R_{ijkl}}
+\phi^i_{\alpha}\phi^j_{\beta}\phi^k_{\bar{\alpha}}\phi^l_{\bar{\beta}}\widehat{R_{ijkl}}
-\phi^i_{\alpha}\phi^j_{\bar{\alpha}}\phi^k_{\bar{\beta}}\phi^l_{\beta}\widehat{R_{ijkl}})\Psi.
\er
By the Bianchi identity, we find
\br
\nonumber
-\phi^i_{\alpha}\phi^j_{\bar{\alpha}}\phi^k_{\bar{\beta}}\phi^l_{\beta}\widehat{R_{ijkl}}
&=&\phi^i_{\alpha}\phi^j_{\bar{\alpha}}\phi^k_{\bar{\beta}}\phi^l_{\beta}(\widehat{R_{iklj}}+\widehat{R_{iljk}})\\
\nonumber &=&\phi^i_{\alpha}\phi^j_{\bar{\beta}}\phi^k_{\beta}\phi^l_{\bar{\alpha}}\widehat{R_{ijkl}}
+\phi^i_{\alpha}\phi^j_{\beta}\phi^k_{\bar{\alpha}}\phi^l_{\bar{\beta}}\widehat{R_{ijkl}}\\
\nonumber &=&-\phi^i_{\bar{\alpha}}\phi^j_{\beta}\phi^k_{\alpha}\phi^l_{\bar{\beta}}\widehat{R_{ijkl}}
+\phi^i_{\alpha}\phi^j_{\beta}\phi^k_{\bar{\alpha}}\phi^l_{\bar{\beta}}\widehat{R_{ijkl}}.
\er
Hence
\br
\nonumber
0&=&4\int_M |\phi^i_{\alpha\bar{\beta}}|^2\Psi-\int_M |\tau(\phi)|^2\Psi+\sqrt{-1}(m-2)\int_M (\phi^i_{\alpha}\phi^i_{\bar{\alpha}0}-\phi^i_{\bar{\alpha}}\phi^i_{\alpha 0})\Psi\\
&&-\sqrt{-1}m\int_M (\phi^i_{\alpha}\phi^i_{\beta}A_{\bar{\alpha}\bar{\beta}}-\phi^i_{\bar{\alpha}}\phi^i_{\bar{\beta}}A_{\alpha\beta})\Psi
-4\int_M \phi^i_{\alpha}\phi^j_{\beta}\phi^k_{\bar{\alpha}}\phi^l_{\bar{\beta}}\widehat{R_{ijkl}}\Psi.
\er
Calculating (\ref{bo7})$\times (m-1)-$(\ref{bo81}) and substituting the result into the above formula, we have
\br
\nonumber
0&=&4\int_M |\phi^i_{\alpha\bar{\beta}}|^2\Psi-\f{1}{m}\int_M |\tau(\phi)|^2\Psi-m\int_M(\phi^i_0)^2\Psi+\f{2(m-1)}{m}\int_M \l\l P\phi+\overline{P\phi}, d_b\phi\r\r\Psi\\
&&-4\int_M \phi^i_{\alpha}\phi^j_{\beta}\phi^k_{\bar{\alpha}}\phi^l_{\bar{\beta}}\widehat{R_{ijkl}}\Psi.
\er
Since
$$
|\phi^i_{\alpha\bar{\beta}}|^2\geq\f{1}{m}|\sum\phi^i_{\alpha\bar{\alpha}}|^2=\f{1}{4m}|\tau(\phi)|^2+\f{m}{4}(\phi^i_0)^2,
$$
we conclude
\be\label{bo12}
-\int_M \l\l P\phi+\overline{P\phi}, d_b\phi\r\r\Psi\geq-\frac{2m}{m-1}\int_M \phi^i_{\alpha}\phi^j_{\beta}\phi^k_{\bar{\alpha}}\phi^l_{\bar{\beta}}\widehat{R_{ijkl}}\Psi.
\ee

\begin{definition}(cf. \cite{Sam2})
A Riemannian manifold $(N^n,h)$ is said to have nonpositive Hermitian curvature if
\be
R_{ijkl}u^i v^j \bar{u^{k}}\bar{v^l}\Psi\leq 0,
\ee
for any complex vectors $u$ and $v$.
\end{definition}

From (\ref{bo12}), we have
\begin{theorem}\label{bot1}
Let $(M^{2m+1},J,\theta)$ be a compact strictly pseudoconvex CR manifold with $m\geq 2$ and $(N,h)$ be a Riemannian manifold with nonpositive Hermitian curvature. Suppose $\phi:M\to N$ is a smooth map, then
$$-\int_M\l P\phi+\overline{P\phi},d_b \phi\r \Psi \geq 0.$$
\end{theorem}

Let's denote
\br
\nonumber Ric(X,Y)&=&R_{\alpha\bar{\beta}}X^{\alpha}Y^{\bar{\beta}},\\
\nonumber Tor(X,Y)&=&\sqrt{-1}(A_{\bar{\alpha}\bar{\beta}}X^{\bar{\alpha}}Y^{\bar{\beta}}-A_{\alpha\beta}X^{\alpha}Y^{\beta}),
\er
where $X=X^{\alpha}T_{\alpha}$, $Y=Y^{\beta}T_{\beta}$ and $R_{\alpha\bar{\beta}}=R^{\gamma}_{\gamma\alpha\bar{\beta}}$ is the pseudo-Hermitian Ricci curvature of $M$. We denote $(\nabla_b\phi^i)_{\mathbb{C}}=\phi^i_{\alpha}T_{\alpha}$.

\begin{theorem}\label{bot2}
Let $(M^{2m+1},J,\theta)$ be a compact strictly pseudoconvex CR manifold with $m\geq 2$ and $(N,h)$ be a Riemannian manifold with nonpositive Hermitian curvature. Let $\phi:M\rightarrow N$ be a pseudoharmonic map.
Suppose that
\be
\big(2Ric-(m+2)Tor\big)(Z,Z)\geq 0,
\ee
for any $Z\in \Gamma^{\infty}(T_{1,0}M)$, then\\
(i) $\phi$ is horizontal totally geodesic, that is $\phi^i_{\alpha\beta}=\phi^i_{\alpha\bar{\beta}}=0$. In particular, $\phi$ is baisc;\\
(ii)If $\big(2Ric-(m+2)Tor\big)(Z,Z)>0$ at one point in $M$, then $\phi$ is constant.
\end{theorem}
\begin{proof}
(i) By (\ref{bo31}), we have
\br
\nonumber
0&=&2\int_M (|\phi^i_{\alpha\beta}|^2+|\phi^i_{\alpha\bar{\beta}}|^2)\Psi-(1+\f{2}{m})\int_M |\tau(\phi)|^2\Psi-\f{4}{m}\int_M \l P\phi+\overline{P\phi},d_b\phi\r\Psi\\
\nonumber &&+\int_M (2Ric-(m+2)Tor)((\nabla_b\phi^i)_{\mathbb{C}},\nabla_b\phi^i)_{\mathbb{C}})\Psi\\
\nonumber&&-2\int_M(\phi^i_{\bar{\alpha}}\phi^j_{\beta}\phi^k_{\alpha}\phi^l_{\bar{\beta}}\widehat{R_{ijkl}}+\phi^i_{\alpha}\phi^j_{\beta}\phi^k_{\bar{\alpha}}\phi^l_{\bar{\beta}}\widehat{R_{ijkl}})\Psi.
\er
By Theorem 4.1, the CR Paneitz operator from $M$ into $N$ is nonnegative. Because of the curvature condition of $N$, the last term of the above formula is nonnegative. Since (4.11) and $\phi$ is pseudoharmonic, we get
$$
0\geq \int_M (|\phi^i_{\alpha\beta}|^2+|\phi^i_{\alpha\bar{\beta}}|^2)\Psi.
$$
Hence $\phi^i_{\alpha\beta}=\phi^i_{\alpha\bar{\beta}}=0$. From $\phi^i_{\alpha\bar{\beta}}=0$, we see that $\phi$ is CR pluriharmonic, so $\phi$ is baisc.

(ii) Since $\phi$ is baisc and pseudoharmonic, by Proposition \ref{bal4}, we have that $\phi$ is harmonic. By the curvature condition of $M$, we have $(\nabla_b\phi^i)_{\mathbb{C}}=0$ in some neigberhood $U$ of that point. Thus we get $\phi$ is constant in $U$. It follows from the unique continuation theorem (cf. \cite{Sam1}) that $\phi$ is constant on $M$.
\end{proof}

\begin{remark}
If the manifold $M$ in Theorem 4.2 is Sasakian and $Ric(Z,Z)\geq 0$, we have $\beta\equiv 0$.
\end{remark}

\section{Basicity of harmonic and pseudoharmonic maps}

Suppose $\phi:(M^{2m+1},J,\theta)\rightarrow (N,h)$ is a smoooth map from a strictly pseudoconvex CR manifold into a Riemannian manifold. We choose the orthonormal CR coframe field $\{\theta,\theta^{\alpha},\theta^{\bar{\alpha}}\}$ on $M$ and the orthonormal coframe field $\{\omega^{i}\}$ on $N$ respectively.
We still use the notaions in Section 3. Set
\begin{eqnarray*}
\theta_{W_2}&=&(\phi^i_0\phi^i_{0\alpha}\theta^{\alpha}+\phi^i_0\phi^i_{0\bar{\alpha}}\theta^{\bar{\alpha}})+\phi^i_0\phi^i_{00}\theta\\
\theta_{W_3}&=&(\phi^i_0\phi^i_{0\alpha}\theta^{\alpha}+\phi^i_0\phi^i_{0\bar{\alpha}}\theta^{\alpha})
\end{eqnarray*}
Clearly $\theta_{W_2}$, $\theta_{W_3}$ are well-defined global $1$-forms on $M$. In fact, $\theta_{W_2}$ is the 1-form corresponding to the vector field $\frac{1}{2}\nabla |d\phi(T)|^2$ and $\theta_{W_3}$ is the 1-form corresponding to the horizontal gradient $\frac{1}{2}\nabla^H|d\phi(T)|^2=\frac{1}{2}\Pi_H \nabla|d\phi(T)|^2$.

By the commutative relations in section 3, we have
\begin{lemma}\label{ssl1}
\begin{eqnarray}\label{ss1}
\label{ss2} \nonumber div\theta_{W_2}&=&2|\phi^i_{0\alpha}|^2+|\phi^i_{00}|^2+\phi^i_0(\phi^i_{\alpha\bar{\alpha}0}+\phi^i_{\bar{\alpha}\alpha0}+\phi^i_{000})
-2\phi^i_0\phi^j_{\alpha}\phi^k_0\phi^l_{\bar{\alpha}}\widehat{R_{ijkl}}\\
&&+2\phi^i_0\phi^i_{\beta}A_{\bar{\beta}\bar{\alpha},\alpha}+2\phi^i_0\phi^i_{\bar{\beta}}A_{\beta \alpha,\bar{\alpha}}+2\phi^i_0\phi^i_{\alpha\beta}A_{\bar{\beta}\bar{\alpha}}+2\phi^i_0\phi^i_{\bar{\alpha}\bar{\beta}}A_{\beta\alpha};\\
\label{ss3} \nonumber div\theta_{W_3}&=&2|\phi^i_{0\alpha}|^2+\phi^i_0(\phi^i_{\alpha\bar{\alpha}0}+\phi^i_{\bar{\alpha}\alpha0})
-2\phi^i_0\phi^j_{\alpha}\phi^k_0\phi^l_{\bar{\alpha}}\widehat{R_{ijkl}}\\
&&+2\phi^i_0\phi^i_{\beta}A_{\bar{\beta}\bar{\alpha},\alpha}+2\phi^i_0\phi^i_{\bar{\beta}}A_{\beta \alpha,\bar{\alpha}}+2\phi^i_0\phi^i_{\alpha\beta}A_{\bar{\beta}\bar{\alpha}}+2\phi^i_0\phi^i_{\bar{\alpha}\bar{\beta}}A_{\beta\alpha}.
\er
\end{lemma}
\begin{remark}
In fact, $div\theta_{W_2}=\frac{1}{2}\Delta |d\phi(T)|^2$, and  $div\theta_{W_3}=\frac{1}{2}\Delta_b |d\phi(T)|^2$.
\end{remark}

\begin{definition}
Let $\phi:M\rightarrow N$ be a smooth map from a strictly pseudoconvex CR manifold $M$ into a Riemannian manifold $N$. The second fundamental form $\beta$ is called split if $\beta(T,X)=0$ for any $X\in H(M)$.
\end{definition}

\begin{remark}
According to (\ref{cr8}), the condition $\beta (T,X)=0$ for $X\in H(M)$ is not equivalent to $\beta (X,T)=0$ for $X \in H(M)$ in general.
From Proposition \ref{pro201} and (\ref{pr8}), it is easy to see that if $\phi$ is baisc, then the second fundamental form $\beta$ is split. The next result shows that if the domain CR manifold is compact, the converse is also true.
\end{remark}

\begin{lemma}\label{ssl2}
Let $\phi:M\rightarrow N$ be a smooth map from a compact strictly pseudoconvex CR manifold $M$ into a Riemannian manifold $N$. If the second fundamental form $\beta$ is split, then $\phi$ is basic.
\end{lemma}
\begin{proof}
By the integration by parts and the commutative formulae (\ref{cr8}), we have
\br\nonumber
0=\sqrt{-1} \int_M (\phi^i_{\alpha}\phi^i_{0\bar{\alpha}}-\phi^i_{\bar{\alpha}}\phi^i_{0\alpha})\Psi
=-\sqrt{-1}\int_M(\phi^i_{\alpha\bar{\alpha}}\phi^i_{0}-\phi^i_{\bar{\alpha}\alpha}\phi^i_{0})\Psi
=m \int_M |\phi^i_0|^2 \Psi.
\er
Thus we have $\phi^i_0=0$, i.e., $d\phi(T)=0$.
\end{proof}

First, we prove the following result of Petit by the moving frame method.
\begin{proposition}\label{ssp1}
(cf. \cite{Pet}) Let $(M^{2m+1},J,\theta)$ be a compact Sasakian manifold and $(N,h)$ be a Riemannian manifold with nonpositive curvature.
Suppose $\phi:M\rightarrow N$ is a harmonic map. Then $\phi$ is basic.
\end{proposition}

\begin{proof}
Since $\phi$ is harmonic, we have $D\tau^{\theta}(\phi)=0$. Consequently, $\phi^i_{\alpha\bar{\alpha}0}+\phi^i_{\bar{\alpha}\alpha0}+\phi^i_{000}=0$.
The Sasakian condition for $M$ means that $A_{\alpha \beta}=0$, for any $\alpha,\beta$, then (\ref{ss2}) becomes
\be
\nonumber div\theta_{W_2}= 2|\phi^i_{0\alpha}|^2+|\phi^i_{00}|^2-2\phi^i_0\phi^j_{\alpha}\phi^k_0\phi^l_{\bar{\alpha}}\widehat{R_{ijkl}}.
\ee
Since the sectional curvature of $N$ is nonpositive, we take $T_{\alpha}=\frac{1}{\sqrt{2}}(e_{\alpha}-iJe_{\alpha})$ and $T_{\bar{\alpha}}=\frac{1}{\sqrt{2}}(e_{\alpha}+iJe_{\alpha})$ and compute the following curvature term to find
\br
\nonumber &&\phi^i_0\phi^j_{\alpha}\phi^k_0\phi^l_{\bar{\alpha}}\widehat{R_{ijkl}}\\
\nonumber&=&h(\widehat{R}(d\phi(T),d\phi(T_{\bar{\alpha}}))d\phi(T_{\alpha}),d\phi(T))\\
\nonumber&=&\f{1}{2}h(\widehat{R}(d\phi(T),d\phi(e_{\alpha}+iJe_{\alpha}))d\phi(e_{\alpha}-iJe_{\alpha}),d\phi(T))\\
\nonumber&=&\f{1}{2}[h(\widehat{R}(d\phi(T),d\phi(e_{\alpha}))d\phi(e_{\alpha}),d\phi(T))+h(\widehat{R}(d\phi(T),d\phi(Je_{\alpha}))d\phi(Je_{\alpha}),d\phi(T))]\\
\nonumber&\leq&0.
\er
Therefore
\be\label{ss6}
div\theta_{W_2}\geq 2|\phi^i_{0\alpha}|^2+|\phi^i_{00}|^2.
\ee
The divergence theorem yields
$$
\phi^i_{00}=\phi^i_{0\alpha}=\phi^i_{0\bar{\alpha}}=0.
$$
The fact that $\phi$ is basic can be easily obtained by Lemma \ref{ssl2}.
\end{proof}

The next result shows that Petit type result is also true for pseudoharmonic maps.

\begin{theorem}\label{ssp2}
Let $(M^{2m+1},J,\theta)$ be a compact Sasakian manifold and $(N,h)$ be a Riemannian manifold with nonpositive curvature.
Suppose $\phi:M\rightarrow N$ is a pseudoharmonic map. Then $\phi$ is basic and harmonic.
\end{theorem}
\begin{proof}
Since $\phi$ is pseudoharmonic, we get
$$
\phi^i_{\alpha\bar{\alpha}0}+\phi^i_{\bar{\alpha}\alpha0}=0.
$$
By (\ref{ss3}), we have
\be\label{ss8}
div \theta_{W_3}\geq 2|\phi^i_{0\alpha}|^2
\ee
Thus $\phi^i_{0\alpha}=\phi^i_{0\bar{\alpha}}=0$. By Lemma \ref{ssl2} again, we get $d\phi(T)=0$.
\end{proof}

\begin{remark}
From Proposition \ref{ssp1} and Theorem \ref{ssp2}, we see that if $M$ is a compact Sasakian manifold and $N$ is a Riemannian manifold with nonpositive curvature, then $\phi :M\to N$ is harmonic if and only if it is pseudoharmonic.
\end{remark}

Now we will use a technique in \cite{PRS} to treat harmonic maps or pseudoharmonic maps from complete noncompact CR manifolds.

\begin{proposition}\label{cnp1}
Let $(M,J,\theta)$ be a complete noncompact Sasakian manifold of dimension $2m+1$ and $(N,h)$ be a Riemannian manifold with nonpositive curvature. Suppose $\phi:M \rightarrow N$ is either a harmonic map or a pseudoharmonic map. If $\phi$ satisfies
\be \label{y601}
(\int_{\partial B_r}|d\phi(T)|^2dS)^{-1}\notin L^1(+\infty),
\ee
where $dS$ is the area volume of $\partial B_r$, then $\phi$ has split second fundamental form $\beta$.
\end{proposition}

\begin{proof}
We consider only the case $\phi$ is a harmonic map, because the other case is analogous.
By the divengence theorem, (\ref{ss6}) gives
\be \label{y602}
\int_{\p B_r} \theta_{W_2}(\frac{\p}{\p r})dS \geq \int_{B_r} (2|\phi^i_{0\alpha}|^2+|\phi^i_{00}|^2)\Psi.
\ee
Recalling the definition of $\theta_{W_2}$ we have
\be \label{y603}
\int_{\p B_r} \theta_{W_2}(\frac{\p}{\p r})dS \leq \{\int_{\p B_r} |\phi^i_0|^2dS\}^{\frac{1}{2}}\{\int_{\p B_r} [2|\phi^i_{0\alpha}|^2+|\phi^i_{00}|^2]dS\}^{\frac{1}{2}}.
\ee
Let
$$
\zeta(r)=\int_{B_r} (2|\phi^i_{0\alpha}|^2+|\phi^i_{00}|^2)\Psi.
$$
Then by the co-area formula, we get
$$
\zeta'(r)=\int_{\p B_r} (2|\phi^i_{0\alpha}|^2+|\phi^i_{00}|^2)dS.
$$
Putting together (\ref{y602}) and (\ref{y603}) and squaring we finally get
\be \label{y604}
\zeta(r)^2\leq (\int_{\p B_r} |\phi^i_0|^2dS)\zeta'(r).
\ee
Next, we reason by contradiction and we suppose $\phi^i_{0\alpha}\neq 0$. It follows that there exists a $R>0$ sufficiently large such that $\zeta(r)>0$, for every $r\geq R$. Fix such an $r$. From (\ref{y604}) we then derive
\be\nonumber
\zeta(R)^{-1}-\zeta(r)^{-1}\geq\int^r_R \frac{dt}{\int_{\p B_t} |\phi^i_0|^2},
\ee
and letting $r\rightarrow +\infty$ we contradict (\ref{y601}).
\end{proof}

\begin{corollary}
Let $(M,J,\theta)$ be a complete noncompact Sasakian manifold of dimension $2m+1$ and $(N,h)$ be a Riemannian manifold with nonpositive curvature. Suppose $\phi:M \rightarrow N$ is either a harmonic map or a pseudoharmonic map. If $\phi$ satisfies
\be \label{y606}
\int_{B_r}|d\phi(T)|^2\Psi\leq Cr^2,
\ee
then
$\phi$ has split second fundamental form $\beta(\phi)$.
\end{corollary}

\begin{proof}
Set
$$
h(r)=\int_{B_r}|d\phi(T)|^2\Psi.
$$
So, by the co-area formula, we have
$$
h'(r)=\int_{\p B_r}|d\phi(T)|^2dS.
$$
From Proposition 3.1 of \cite{RS}, we know that
$$
\f{r}{h(r)}\notin L^1(+\infty)  \quad \mbox{implies \quad $\f{1}{h'(r)}\notin L^1(+\infty)$}.
$$
Suppose that $\phi$ satisfies (\ref{y606}), this implies
$$
\f{r}{h(r)}\notin L^1(+\infty).
$$
Thus we deduce $\f{1}{h'(r)}\notin L^1(+\infty)$, that is, $\phi$ satisfies (\ref{y601}). Hence we prove the corollary.
\end{proof}

\begin{proposition}\label{cnl1}
Let $\phi:(M^{2m+1},J,\theta)\rightarrow (N,h)$ be a smooth map from a complete noncompact strictly pseudoconvex CR manifold $M$ into a Riemannian manifold $N$. If the second fundamental form $\beta$ is split and
\be\label{com12}
(\int_{\p B_r} e_H(\phi)dS)^{-1}\notin L^1(+\infty),
\ee
then $\phi$ is basic.
\end{proposition}
\begin{proof}
Since $\phi$ has split second fundamental form $\beta$, we have
\begin{eqnarray*}
m\int_{B_r} |\phi^i_0|^2\Psi&=&-\sqrt{-1} \int_{B_r} div(\phi^i_0\phi^i_{\alpha}\theta^{\alpha}-\phi^i_0\phi^i_{\bar{\alpha}}\theta^{\bar{\alpha}})\Psi\\
&\leq&2 \{\int_{\partial B_r} |\phi^i_{0}|^2dS\}^{1/2} \{\int_{\partial B_r} |\phi^i_{\alpha}|^2dS\}^{1/2}.
\end{eqnarray*}
Set $\eta(r)=\int_{B_r} |\phi^i_0|^2\Psi$. Then we have
$$
\frac{m^2}{4}\eta(r)^2\leq (\int_{B_r} e_H(\phi)\Psi)\eta'(r).
$$
If $\phi$ is not basic, then for $r>R$,
$$
\eta(R)^{-1}-\eta(r)^{-1}\geq \int^r_R \frac{dt}{\int_{\p B_t}e_H(\phi)dS},
$$
where $R$ is large enough such that $\eta(R)>0$, and letting $r\rightarrow +\infty$ we contradict (\ref{com12}).
\end{proof}

\begin{theorem}\label{sst2}
Let $(M^{2m+1},J, \theta)$ be a complete noncompact Sasakian manifold and $(N,h)$ be a Riemannian manifold with nonpositive curvature. Suppose $\phi: M\to N$ is either a harmonic map or a pseudoharmonic map. If $\phi$ satisfies
\be \label{ss14}
(\int_{\partial B_r}e(\phi)dS)^{-1}\notin L^1(+\infty),
\ee
where $e(\phi)=\frac{1}{2}trace_{g_{\theta}}(\phi^*h)$ is the energy density of $\phi$, then $\phi$ is a basic map.
\end{theorem}

\begin{proof}
Since $e(\phi)=\frac{1}{2}|d\phi(T)|^2+e_H(\phi)$, the condition (\ref{ss14}) implies both (\ref{y601}) and (\ref{com12}). It follows from Proposition \ref{cnp1} and \ref{cnl1} that $\phi$ is basic.
\end{proof}

\begin{corollary}
Let $(M,J,\theta)$ be a complete noncompact Sasakian manifold of dimension $2m+1$ and $(N,h)$ be a Riemannian manifold with nonpositive curvature. Suppose $\phi:M \rightarrow N$ is either a harmonic map or a pseudoharmonic map. If $\phi$ satisfies
\be
\int_{B_r}e(\phi)\Psi\leq Cr^2,
\ee
then $\phi$ is basic.
\end{corollary}

\section{CR pluriharmonicity of harmonic and pseudoharmonic\\
maps}

In this section, we give some conditions to ensure the CR pluriharmonicity for both harmonic and pseudoharmonic maps from either a compact Sasakian manifold or a complete Sasakian manifold. Recall that Petit \cite{Pet} gave similar results for harmonic maps from a compact Sasakian manifold by using tools of Spinorial geometry, although he didn't mention the notion of CR pluirharmonicity. The moving frame method, which enables us to treat both cases of harmonic maps and pseudoharmonic maps, seems more closer to the classical methods in differential geometry. Inspired by Sampson's technique (cf. also \cite{Don1}), we introduce
\be
\theta_{W_4}=(\phi^i_{\alpha}\phi^i_{\bar{\alpha}\beta}\theta^{\beta}+\phi^i_{\bar{\alpha}}\phi^i_{\alpha\bar{\beta}}\theta^{\bar{\beta}}).
\ee
Note that $\theta_{W_4}$ consists of partial terms of $\theta_{W_1}$.
\begin{lemma}
\br
\label{ss4} div\theta_{W_4}&=&2|\phi^i_{\alpha\bar{\beta}}|^2+\phi^i_{\alpha}\phi^i_{\beta\bar{\beta}\bar{\alpha}}
+\phi^i_{\bar{\alpha}}\phi^i_{\bar{\beta}\beta\alpha}
-2\phi^i_{\alpha}\phi^j_{\beta}\phi^k_{\bar{\alpha}}\phi^l_{\bar{\beta}}\widehat{R_{ijkl}}\nonumber\\
&&-\sqrt{-1}(m-1)(\phi^i_{\alpha}\phi^i_{\beta}A_{\bar{\alpha}\bar{\beta}}-\phi^i_{\bar{\alpha}}\phi^i_{\bar{\beta}}A_{\alpha \beta})-\sqrt{-1}(\phi^i_{\alpha}\phi^i_{0\bar{\alpha}}-\phi^i_{\bar{\alpha}}\phi^i_{0 \alpha}).
\er
\end{lemma}
\begin{proof}
Since the computation for deriving (\ref{ss4}) is similar to that in Lemma 4.1, we omit its details.
\end{proof}

\begin{theorem}\label{sst1}
Let $(M,J,\theta)$ be a compact Sasakian manifold of dimension $2m+1$ and $(N,h)$ be a Riemannian manifold with nonpositive Hermitian curvature. Suppose $\phi:M\rightarrow N$ is either a harmonic map or a pseudoharmonic map. Then $\phi$ is CR pluriharmonic and
\be\label{y514}
\phi^i_{\alpha}\phi^j_{\beta}\phi^k_{\bar{\alpha}}\phi^l_{\bar{\beta}}\widehat{R_{ijkl}}=0.
\ee
\end{theorem}
\begin{proof}
Since $N$ has a nonpositive Hermitian curvature, the sectional curvature is nonpositive. According to Proposition \ref{ssp1} and Theorem \ref{ssp2}, we know that the conditon that $\phi$ is harmonic is equivalent to that $\phi$ is pseudoharmonic. Besides, the map is basic in this circumstance.
By (\ref{cr8}), we have $\phi^i_{\alpha\bar{\beta}}=\phi^i_{\bar{\beta}\alpha}$ for any $\alpha,\beta$. Then we obtain $\tau(\phi)=2\phi^i_{\beta\bar{\beta}}E_i$ and $\phi^i_{\bar{\beta}\beta\alpha}=\phi^i_{\beta\bar{\beta}\alpha}$.

By (\ref{ss4}) and the fact that $M$ is Sasakian, we get
\br
\nonumber div\theta_{W_4}&=&2|\phi^i_{\alpha\bar{\beta}}|^2+\phi^i_{\alpha}\phi^i_{\beta\bar{\beta}\bar{\alpha}}
+\phi^i_{\bar{\alpha}}\phi^i_{\bar{\beta}\beta\alpha}-2\phi^i_{\alpha}\phi^j_{\beta}\phi^k_{\bar{\alpha}}\phi^l_{\bar{\beta}}\widehat{R_{ijkl}}\\
\nonumber &=&2|\phi^i_{\alpha\bar{\beta}}|^2+\f{1}{2}\l\l d_b \phi, \nabla_b \tau(\phi)\r\r-2\phi^i_{\alpha}\phi^j_{\beta}\phi^k_{\bar{\alpha}}\phi^l_{\bar{\beta}}\widehat{R_{ijkl}}\\
\label{plu64}&=&2|\phi^i_{\alpha\bar{\beta}}|^2-2\phi^i_{\alpha}\phi^j_{\beta}\phi^k_{\bar{\alpha}}\phi^l_{\bar{\beta}}\widehat{R_{ijkl}}.
\er
Since $N$ has nonpositive Hermitian curvature, we have
$$
\phi^i_{\alpha}\phi^j_{\beta}\phi^k_{\bar{\alpha}}\phi^l_{\bar{\beta}}\widehat{R_{ijkl}}\leq 0.
$$
By the divergence theorem, we derive from (\ref{plu64}) that $\phi$ is a CR pluriharmonic map with property (\ref{y514}).
\end{proof}

Let $(N^n,h)$ be a K\"{a}hler manifold. The curvature operator $Q$ of $N$ is defined by
\be
\nonumber \l Q(X\wedge Y), Z\wedge W\r=\l R(X,Y)W,Z  \r
\ee
for any $X,Y,Z,W\in TM$.
The complex extension of $Q$ to $\wedge^2T^{\mathbb{C}}N$ is also denoted by $Q$.  We introduce
\be
\nonumber \ll Q(X\wedge Y), Z\wedge W \gg=\l Q(X\wedge Y), \overline{Z\wedge W} \r.
\ee
The K\"{a}hler identity of $N$ yields
$$
Q|_{\wedge^{(2,0)}T^{\mathbb{C}}N}=Q|_{\wedge^{(0,2)}T^{\mathbb{C}}N}=0.
$$
Set
$$
Q^{(1,1)}=Q: \wedge^{(1,1)}T^{\mathbb{C}}N \rightarrow \wedge^{(1,1)}T^{\mathbb{C}}N.
$$

\begin{definition}(cf. \cite{Siu1})
Let $(N^n,h)$ be a K\"{a}hler manifold. The curvature tensor of $(N,h)$ is said to be strongly negative (resp. strongly semi-negative) if
\be
\nonumber \ll Q^{(1,1)}(\xi), \xi \gg=\l Q^{(1,1)}(\xi), \overline{\xi} \r <0 \quad \mbox{(resp. $\leq 0$)}
\ee
for any $\xi=(Z\wedge W)^{(1,1)}\neq 0$, $Z,W\in \Gamma^{\infty}(TN^{\mathbb{C}})$.
\end{definition}

\begin{remark}
By comparing the Definitions 4.1 and 6.1, we find that the notions of nonpositive Hermitian curvature and strongly semi-negative curvature are equivalent for K\"{a}hler manifolds. However, we should point out that one cannot introduce the notion of negative Hermitian curvature for K\"ahler manifolds due to the K\"ahler identity.
\end{remark}

Let
\be
\theta_{W_5}=\phi^{\bar{i}}_{\alpha}\phi^i_{\bar{\alpha}\beta}\theta^{\beta}+\phi^i_{\bar{\alpha}}\phi^{\bar{i}}_{\alpha\bar{\beta}}\theta^{\bar{\beta}}.
\ee
Then we have
\br
\nonumber div \theta_{W_5}&=&2|\phi^i_{\alpha\bar{\beta}}|^2+\phi^{\bar{i}}_{\alpha}\phi^i_{\beta\bar{\beta}\bar{\alpha}}
+\phi^i_{\bar{\alpha}}\phi^{\bar{i}}_{\bar{\beta}\beta\alpha}-\ll Q(\phi_{\alpha}\wedge\phi_{\beta}), \phi_{\alpha}\wedge\phi_{\beta} \gg\\
\label{plu66}&&-\sqrt{-1}(m-1)(\phi^{\bar{i}}_{\alpha}\phi^i_{\beta}A_{\bar{\alpha}\bar{\beta}}-\phi^i_{\bar{\alpha}}\phi^{\bar{i}}_{\bar{\beta}}A_{\alpha\beta})
-\sqrt{-1}(\phi^{\bar{i}}_{\alpha}\phi^i_{0\bar{\alpha}}-\phi^i_{\bar{\alpha}}\phi^{\bar{i}}_{0\alpha}).
\er

\begin{theorem}\label{sst4}
Let $\phi:(M^{2m+1},J,\theta)\rightarrow (N,h)$ be a harmonic or pseudoharmonic map from a compact Sasakian manifold into a K\"{a}hler manifold with strongly semi-negative curvature. Then $\phi$ is a CR pluriharmonic map and
\be\label{plu65}
\l\l Q(\phi_{\alpha}\wedge\phi_{\beta}), \phi_{\alpha}\wedge\phi_{\beta} \r\r=0,
\ee
where $\phi_{\alpha}=d\phi(T_{\alpha})$.
\end{theorem}

\begin{proof}
Since strongly semi-negative curvature implies non-positive sectional curvature, we get that $\phi$ must be pseudoharmonic and basic. Then we have $\phi^i_{\alpha\bar{\beta}}=\phi^i_{\bar{\beta}\alpha}$ and $\phi^i_{0\bar{\alpha}}=\phi^i_{0\alpha}=0$. So we get $\tau(\phi)=2(\phi^i_{\beta\bar{\beta}}E_i+\phi^{\bar{i}}_{\beta\bar{\beta}}E_{\bar{i}})=0$, i.e., $\phi^i_{\beta\bar{\beta}}=\phi^{\bar{i}}_{\beta\bar{\beta}}=0$.
As $M$ is Sasakian, by (\ref{plu66}) we have
\br\label{plu67}
div \theta_{W_5}&=&2|\phi^i_{\alpha\bar{\beta}}|^2-\ll Q(\phi_{\alpha}\wedge\phi_{\beta}), \phi_{\alpha}\wedge\phi_{\beta} \gg.
\er
The divergence theorem implies $\phi$ is CR pluriharmonic and $\l\l Q(\phi_{\alpha}\wedge\phi_{\beta}), \phi_{\alpha}\wedge\phi_{\beta} \r\r=0$.
\end{proof}

Now we attempt to give some conditions to ensure CR pluriharmonicity for harmonic and pseudoharmonic maps from complete noncompact Sasakian manifolds.
\begin{theorem}\label{comt1}
Let $(M,J,\theta)$ be a complete noncompact Sasakian manifold and $(N,h)$ be a Riemannian manifold with nonpositive Hermitian curvature. Suppose $\phi:M \rightarrow N$ is either a harmonic map or a pseudoharmonic map. If $\phi$ satisfies
\be \label{y608}
(\int_{\partial B_r} e(\phi)dS)^{-1}\notin L^1(+\infty),
\ee
then $\phi$ is a CR pluriharmonic map with the property (\ref{y514}).
\end{theorem}

\begin{proof}
By Theorem \ref{sst2}, we get that $\phi$ is basic. Under the conditions in the theorem, by (\ref{ss4}) we have
$$div \theta_{W_4}\geq2|\phi^i_{\alpha\bar{\beta}}|^2.$$
Using the divergence theorem, we get
\be\label{y609}
\int_{\partial B_r} \theta_{W_4}(\frac{\partial}{\partial r})dS\geq 2\int_{B_r}|\phi^i_{\alpha\bar{\beta}}|^2\Psi.
\ee
On the other hand, by the definition of $\theta_{W_4}$, we have
\be\label{y610}
\int_{\partial B_r} \theta_{W_4}(\frac{\partial}{\partial r})dS \leq 2\{\int_{\p  B_r}e_H(\phi)dS\}^{\frac{1}{2}}\{\int_{\p B_r} |\phi^i_{\alpha\bar{\beta}}|^2\Psi\}^{\frac{1}{2}}.
\ee
Putting together (\ref{y609}) and (\ref{y610}) and squaring we finally get
\be\label{y611}
\gamma(r)^2\leq (\int_{\p B_r} e_H(\phi)dS) \gamma'(r),
\ee
where we have set
$$
\gamma(r)=\int_{B_r} |\phi^i_{\alpha\bar{\beta}}|^2\Psi.
$$
Next suppose that $\phi$ is not CR pluriharmonic. Then there exists a $R>0$ sufficiently large such that $\gamma(R)>0$. For any $r\geq R$, from (\ref{y611}) we can deduce
\be \nonumber
\gamma(R)^{-1}-\gamma(r)^{-1}\geq\int^r_R \frac{dt}{\int_{\p B_t} e_H(\phi)},
\ee
and letting $r\rightarrow +\infty$ we contradict (\ref{y608}). Hence $\phi$ is CR pluriharmonic. By definition, we have $\theta_{W_4}\equiv 0$. Then (\ref{ss4}) implies that $\phi$ satisfies (\ref{y514}).
\end{proof}

\begin{corollary}
Let $(M,J,\theta)$ be a complete noncompact Sasakian manifold and $(N,h)$ be a Riemannian manifold with nonpositive Hermitian curvature. Suppose $\phi:M \rightarrow N$ is either a harmonic map or a pseudoharmonic map. If $\phi$ satisfies
\be
\nonumber \int_{B_r}e(\phi)\Psi \leq Cr^2,
\ee
then $\phi$ is a CR pluriharmonic map with the property (\ref{y514}).
\end{corollary}

\begin{theorem}
Let $\phi:(M^{2m+1},J,\theta)\rightarrow (N,h)$ be a harmonic or pseudoharmonic map from a complete noncompact Sasakian manifold into a K\"{a}hler manifold with strongly semi-negative curvature. If $\phi$ satisfies
\be \label{plu612}
(\int_{\partial B_r} e(\phi)dS)^{-1}\notin L^1(+\infty),
\ee
then $\phi$ is a CR pluriharmonic map with the property (\ref{plu65}).
\end{theorem}
\begin{proof}
Obviously, the map $\phi$ is basic, and hence $\phi^i_{\alpha\bar{\beta}}=\phi^i_{\bar{\beta}\alpha}$. It follows from (\ref{plu67}) and the divergence that
$$2\int_{B_r}|\phi^i_{\alpha\bar{\beta}}|^2\Psi \leq \int_{B_r}div \theta_{W_5}\Psi=\int_{\partial B_r}\theta_{W_5}(\frac{\partial}{\partial r})dS. $$
By the definition of $\theta_{W_5}$, we have
$$
\int_{\p B_r}\theta_{W_5}(\frac{\p}{\p r})dS \leq 2\{\int_{\p B_r}|\phi^i_{\bar{\alpha}}|^2dS\}^{1/2}\{\int_{\p B_r}|\phi^i_{\bar{\alpha}\beta}|^2dS\}^{1/2}.
$$
Set
$$\rho(r)=\int_{B_r}|\phi^i_{\bar{\alpha}\beta}|^2\Psi.$$
Then
\be\label{plu613}
\rho(r)^2\leq\rho'(r)(\int_{\p B_r}|\phi^i_{\bar{\alpha}}|^2dS).
\ee
Suppose that $\phi$ isn't CR pluriharmonic, then there exists a $R>0$ sufficiently large such that $\rho(r)>0$ for any $r>R$. Fix such a $R$. From (\ref{plu613}) we deduce the following
$$
\rho(R)^{-1}-\rho(r)^{-1}\geq \int_R^r \frac{dt}{\int_{\p B_r}|\phi^i_{\bar{\alpha}}|^2},
$$
and letting $r\rightarrow +\infty$ we contradict (\ref{plu612}). Hence $\phi$ is CR pluriharmonic. By definition, we get that $\theta_{W_5}\equiv 0$. Then (\ref{plu66}) implies that $\phi$ satisfies (\ref{plu65}).
\end{proof}

\begin{corollary}
Let $\phi:(M^{2m+1},J,\theta)\rightarrow (N,h)$ be a harmonic or pseudoharmonic map from a complete noncompact Sasakian manifold into a K\"{a}hler manifold with strongly semi-negative curvature. If $\phi$ satisfies
\be
\nonumber \int_{B_r}e(\phi)\Psi \leq Cr^2,
\ee
then $\phi$ is a CR pluriharmonic map with the property (\ref{plu65}).
\end{corollary}

\section{Siu-Sampson type results}

In this section, we will  establish some results of Siu-Sampson type for both harmonic maps and pseudoharmonic maps from compact Sasakian manifolds. Similar to the results for harmonic maps from K\"ahler manifolds in \cite{CT, Sam2, Siu1}, we may derive CR holomorphicity under rank conditions for harmonic and pseudoharmonic maps from compact Sasakian manifolds by analysing the curvature equations (\ref{plu65}).  Note that Petit \cite{Pet} also gave the CR holomorphicity results for harmonic maps from Sasakian manifolds using spinorial geometry. As mentioned previously, our method is different from his. Besides recapturing Petit's results by using the moving frame method, we also add some new results which include the results for pseudoharmonic maps, the conic extension of harmonic maps from Sasakian manifolds and a unique continuation theorem for CR holomorphicity.

Suppose now that the target manifold N is a locally symmetric space of noncompact type. Then the universal covering manifold of N is a symmetric space G/K, where K is a connected and closed subgroup of the noncompact connected Lie group $G$, and $G/K$ is given the invariant metric determined by the Killing form $\l,\r$ on $\mathfrak{g}$. If the corresponding Cartan decomposition of the Lie algebra of $G$ is $\mathfrak{g}=\mathfrak{k}+\mathfrak{p}$, then
the real tangent space of $N$ at any point can be identified with $\mathfrak{p}$. The curvature tensor of $N$ is given by
$$\tilde{R}(X,Y)Z=-[[X,Y],Z],$$
for any $X,Y,Z \in \mathfrak{p}$, and the Hermitian curvature of $N$ is given by
\be
\l \tilde{R}(X,Y)\overline{Y},\overline{X}\r=\l [X,Y],[\overline{X},\overline{Y}]\r.
\ee
Therefore, (\ref{y514}) yields that
\be
[d\phi(T_{\alpha}),d\phi(T_{\beta})]=0,
\ee
for any $\alpha,\beta$. In this way, we get

\begin{proposition}\label{y52}
Let $(M,J,\theta)$ be a compact Sasakian manifold and $N$ a locally symmetric space of noncompact type. If $\phi: M\rightarrow N$ is either a harmonic map or a pseudoharmonic map, then $\phi$ is CR pluriharmonic and for any $x\in M$, $d\phi_x$ maps $T_{1,0}M_x$ onto an abelian subspace $W$ of $\mathfrak{p}\otimes\mathbb{C}$.
\end{proposition}

Under the assumption of Proposition \ref{y52}, the image under $d\phi_x$ of real tangent space $T_xM$ is the subspace of real points of space $W+\overline{W}\subset T_{\phi(x)}^{\mathbb{C}}N$, so that
$$dim_{\mathbb{R}}d\phi_x(T_xM)=dim_{\mathbb{C}}(W+\overline{W})\leq 2 dim_{\mathbb{C}}W.$$
Hence we obtain the following estimate:
\be\label{y520}
rank_{\mathbb{R}}(d\phi)\leq 2 max\{dim_{\mathbb{C}} W| W\subset \mathfrak{p}\otimes\mathbb{C}, [W,W]=0\}.
\ee
When $G=SO(1,n)$, then $dim W \leq 1$ (cf. \cite{Sam2}). Thus we get the following result.
\begin{corollary}
Let $(M,J,\theta)$ be a compact Sasakian manifold and $N$ a manifold of constant negative curvature. If $\phi: M\rightarrow N$ is harmonic or pseudoharmonic, then $rank_{\mathbb{R}}(d\phi)\leq 2$.
\end{corollary}

If $G/K$ is a Hermitian symmetric space, then corresponding to any invariant complex structure on $G/K$ we have the decomposition $$\mathfrak{p}\otimes \mathbb{C}=\mathfrak{p}^{1,0}\oplus\mathfrak{p}^{0,1},$$
and the integrability condition $[\mathfrak{p}^{1,0}, \mathfrak{p}^{1,0}]\subset \mathfrak{p}^{1,0}$ is equivalent, in view of $[\mathfrak{p},\mathfrak{p}]\subset\mathfrak{k}$, to $[\mathfrak{p}^{1,0}, \mathfrak{p}^{1,0}]=0$, thus $\mathfrak{p}^{1,0}$ is an abelian subalgebra of $\mathfrak{p}\otimes \mathbb{C}$.
\begin{lemma}(cf. \cite{CT})\label{y53}
Let $G/K$ be a symmetric space of non-compact type. Let $W\subset \mathfrak{p}\otimes\mathbb{C}$ be an abelian subspace. Then $dim W\leq\frac{1}{2}dim\mathfrak{p}\otimes\mathbb{C}$. Equality holds in this inequality if and only if $G/K$ is Hermitian symmetric and $W=\mathfrak{p}^{1,0}$ for any invariant complex structure on $G/K$.
\end{lemma}

From (\ref{y520}) and Lemma \ref{y53}, we get immediately the following result.
\begin{corollary}
Let $\phi: M\rightarrow N$ be as in Proposition \ref{y52} and suppose that $N$ is not locally Hermitian symmetric. Then $rank d\phi<dim N$.
\end{corollary}

The above corollary use only the case of strict inequality in Lemma \ref{y53}. We have treated the case of equality in such detail in order to obtain the following theorem.
\begin{theorem}
Let $(M,J,\theta)$ be a compact Sasakian manifold and $N$ a locally Hermitian symmetric space of noncompact type whose universal cover does not contain the hyperbolic plane as a factor. If $\phi: M\rightarrow N$ is either a harmonic map or a pseudoharmonic map, and there is a point $x\in M$ such that $d\phi(T_xM)=T_{\phi(x)}N$, then $\phi$ is CR holomorphic.
\end{theorem}
\begin{proof}
Since $d\phi(T_{1,0}M)$ is an abelian subspace of half the dimension, it must be $\mathfrak{p}^{1,0}$ for an invariant complex structure on $N$, i.e., $d\phi_x(T_{1,0}M_x)=\mathfrak{p}^{1,0}$. Consequently this property must hold on a neighborhood $U$ of $x$. By Proposition \ref{y52} and Proposition \ref{bap3}, we have $d\phi(T)=0$. Therefore, the map $\phi$ is CR holomorphic on $U$. We get that the map $\phi$ is CR holomorphic on $M$ by the following unique continuation Proposition \ref{sst6}.
\end{proof}

Now, we will give some fundamental knowledge about the warped product. Let $(B,g_B)$ and $(S,g_S)$ be two Riemannian manifolds and $f$ be a positive smooth function on $B$. Consider the product manifold $B\times S$ with its natural projections $\pi_B:B\times S\rightarrow B$ and $\pi_S:B\times S\rightarrow S$. The warped product $B\times_fS$ is the manifold $B\times S$ furnished with the following Riemannian metric
\be
\tilde{g}=\pi_B^*(g_B)+(f\circ\pi_B)^2\pi_S^*(g_S).
\ee
The Levi-Civita connection of $N=B\times_fS$ can now be related to those of $B$ and $S$ as follows.
\begin{lemma}(cf. \cite[p.~206]{O'N})\label{ssl4}
Let $\tilde{\nabla}$, $^B\nabla$ and $^S\nabla$ be the Levi-Civita connections on $N$, $B$ and $S$ respectively. If $X$, $Y$ are vector fields on $S$ and $V$,$W$ are vector fields on $B$, the lift of $X,Y,V,W$ to $B\times_fS$ is also denoted by the same notations, then\\
(i) $\widetilde{\nabla}_V W$ is the lift of $^B\nabla_V W$\\
(ii) $\w{\nabla}_VX=\w{\nabla}_XV=\f{Vf}{f}X$;\\
(iii) $(\w{\nabla}_XY)_B=-(\tilde{g}(X,Y)/f) grad f$;\\
(iv) $(\w{\nabla}_XY)_S$ is the lift of $^S\nabla_XY$ on $S$.
\end{lemma}

Now we consider the special case: let $(M,\theta,J)$ be a strictly pseudoconvex CR manifold and $C(M)$ be the manifold $\mathbb{R}^+\times_r M$ endowed with the metric $\tilde{g}=dr^2+\frac{r^2}{4}g_{\theta}$. Therefore, by Lemma \ref{ssl4}, we have
\be\label{y524}
\w{\nabla}_{\f{\p}{\p r}}\f{\p}{\p r}=0, \quad \w{\nabla}_{\f{\p}{\p r}} X=\w{\nabla}_X\f{\p}{\p r}=\f{1}{r}X, \quad \w{\nabla}_XY=\nabla^{\theta}_XY-\frac{1}{4}g_{\theta}(X,Y)r\f{\p}{\p r}.
\ee
\begin{proposition}(cf. \cite{BG})
If $(M,J,\theta)$ is a Sasakian manifold, then $(C(M),\tilde{g})$ is K\"{a}hler.
\end{proposition}
\begin{proof}
Set $\zeta=\frac{r}{2}\frac{\p}{\p r}$ and define smooth section of End$TC(M)$ by the formula
\be\label{y525}
\tilde{J}Y=JY-\theta(Y)\zeta, \quad \tilde{J}\zeta=T.
\ee
It is easy to see that $\tilde{J}$ is an almost complex structure on $C(M)$ and the metric $\tilde{g}$ is Hermitian. From (\ref{y524}) and (\ref{y525}) we can show that $\w{\nabla}\tilde{J}=0$. Thus $C(M)$ is K\"{a}hler.
\end{proof}

By (\ref{pr4}), (\ref{y524}) and (\ref{y525}), we can derive the following Lemmas \ref{y526}, \ref{lem4} and \ref{y528}.
\begin{lemma}\label{y526}
Let $(M^{2m+1},J,\theta)$ be a Sasakian manifold, $(C(M),\tilde{g})$ its cone manifold, $(N^n,h)$ a Riemannian manifold. If $\phi:M\rightarrow N$ is a harmonic map, then the conic extension $\tilde{\phi}:C(M)\rightarrow N$ defined by
\be \label{y527}
\tilde{\phi}(x,r)=\phi(x)
\ee
is also harmonic.
\end{lemma}

\begin{lemma}\label{lem4}
Let $(M^{2m+1},J,\theta)$ be a Sasakian manifold, $(C(M),\tilde{g})$ its cone manifold, $(N,h)$ a Riemannian manifold. If $\phi:M\rightarrow N$ is a CR pluriharmonic map, then the conic extension $\tilde{\phi}$ is a pluriharmonic map.
\end{lemma}

\begin{lemma}\label{y528}
Let $\phi:(M,J,\theta)\rightarrow (N,h,J')$ be a smooth map from a Sasakian manifold to a K\"{a}hler manifold, $(C(M),\tilde{g})$ the cone manifold of $M$, the conic extension of $\phi$ is defined by (\ref{y527}). Then $\phi$ is a CR holomorphic (resp. CR anti-holomorphic) map if and only if $\tilde{\phi}$ is holomorphic (resp. anti-holomorphic).
\end{lemma}

In \cite{Siu1}, Siu derived the following unique continuation theorem for holomorphicity.
\begin{lemma}(cf. \cite{Siu1})\label{y529}
Suppose $M,N$ are two K\"{a}hler manifolds and $\phi:M\rightarrow N$ is a harmonic map. Let $U$ be a nonempty open subset of $M$. If $\phi$ is holomorphic (resp. anti-holomorphic) on $U$, then $\phi$ is holomorphic (resp. anti-holomorphic) on $M$.
\end{lemma}

From the Lemmas \ref{y526}, \ref{y528} and \ref{y529}, we get the following unique continuation theorem.
\begin{proposition}\label{sst6}
Let $\phi: (M^{2m+1},J,\theta)\rightarrow(N,h)$ be a harmonic map from a connected Sasakian manifold to a K\"{a}hler manifold. Let $U$ be a nonempty open subset of $M$. If $\phi$ is CR holomorphic (resp. CR anti-holomorphic) on $U$, then $\phi$ is CR holomorphic (resp. CR anti-holomorphic ) on $M$.
\end{proposition}

\begin{proof}
From Lemma \ref{y526}, we know that $\tilde{\phi}:C(M)\rightarrow N$ is harmonic. Suppose $\phi$ is CR holomorphic on $U$. It follows from Lemma \ref{y528} that $\tilde{\phi}$ is holomorphic on $\mathbb{R}_+\times_r U$. Using Lemmas \ref{y528} and \ref{y529}, we conclude that $\phi$ is CR holomorphic on $M$.
\end{proof}

Now we may establish the following results.
\begin{theorem}
\label{ssc3}
Let $(M^{2m+1},J,\theta)$ be a compact Sasakian manifold and $N$ be a K\"{a}hler manifold with strongly negative curvature. Suppose $\phi:M\rightarrow N$ is either a harmonic map or a pseudoharmonic map, and $rank_{\mathbb{R}}d\phi\geq 3$ at some point of $M$, then $\phi$ is CR holomorphic or CR anti-holomorphic on $M$.
\end{theorem}

\begin{proof}
From Theorem \ref{sst4} and Lemma \ref{y526}, we know that $\tilde{\phi}$ is harmonic. By Siu's results, we have $\tilde{\phi}$ is $\pm$holomorphic on $C(M)$. By Proposition \ref{sst6}, we conclude that $\phi$ is CR $\pm$holomorphic on $M$.
\end{proof}

Keeping in mind Udagawa's proof to Theorem 4 of \cite{Uda} the following result is relevant.
 
\begin{theorem}
Every CR pluriharmonic map $\phi:(M,J,\theta)\rightarrow (N,h)$ from a Sasakian manifold $M$ into an irreducible Hermitian symmetric space $N$ of compact or noncompact type is CR $\pm$holomorphic if $Max_M rank_{\mathbb{R}}d\phi\geq 2P(N)+1$, where $P(N)$ is the degree of strong non-degenerate of the bisectional curvature of $N$ (cf. \cite{Siu2} for the definition of  the degree of strong non-degenerate of the bisectional curvature of $N$).
\end{theorem}

\begin{proof}
By Lemma \ref{lem4}, we have $\tilde{\phi}$ is pluriharmonic. Since $Max_M rank_{\mathbb{R}}d\phi\geq 2P(N)+1$ implies that $Max_{C(M)} rank_{\mathbb{R}}d\tilde\phi\geq 2P(N)+1$, by Theorem 4 of \cite{Uda} we get that $\tilde{\phi}$ is $\pm$holomorphic. From Lemma \ref{y528}, we prove that $\phi$ is CR $\pm$holomorphic.
\end{proof}

\section*{Acknowledgments }
This work was partially supported by the National Natural Science Foundation of China [grant number 11271071] and Laboratory of Mathematics for Nonlinear Science, Fudan; research of the last author was partially supported by HUST Innovation Research Grant 0118011034.

\begin{flushleft}
Tian Chong and Yibin Ren\\
School of Mathematical Science\\
Fudan University, Shanghai 200433, P.R. China\\
E-mail address: valery4619@sina.com (Tian Chong)\\
E-mail address: allenrybqqm@hotmail.com (Yibin Ren)
\end{flushleft}

\begin{flushleft}
Yuxin Dong\\
School of Mathematical Science\\
and\\
Laboratory of Mathematics for Nonlinear Science\\
Fudan University, Shanghai 200433, P.R. China\\
E-mail address: yxdong@fudan.edu.cn
\end{flushleft}

\begin{flushleft}
Guilin Yang\\
School of Mathematics and Statistics\\
Huazhong University of Science and Technology\\
Wuhan 430074, P.R. China\\
E-mail address: glyang@hust.edu.cn
\end{flushleft}

\end{document}